\documentclass[a4paper,oneside]{article}

\usepackage[latin1]{inputenc}
\usepackage{amsmath, amsthm, amsfonts, amssymb}
\usepackage[all]{xy}
\usepackage{graphicx}
\usepackage{MnSymbol}

\numberwithin{equation}{section}

\newtheorem{theorem}{Theorem}[subsection]
\newtheorem{corollary}[theorem]{Corollary}

\newtheorem{proposition}[theorem]{Proposition}

\theoremstyle{definition}

\newtheorem{example}[theorem]{Example}
\newtheorem{remark}[theorem]{Remark}

\newcommand{\U}{{\mathcal U}}
\newcommand{\R}{\mathbb R}
\newcommand{\Z}{\mathbb Z}
\newcommand{\id}{{\rm id}}
\newcommand{\pr}{{\rm pr}}
\newcommand{\xto}{\xrightarrow}
\newcommand{\action}{\curvearrowright}
\newcommand{\raction}{\curvearrowleft}
\newcommand{\toto}{\rightrightarrows}
\newcommand{\then}{\Rightarrow}

\newcommand{\from}{\leftarrow}
\newcommand{\xfrom}{\xleftarrow}
\newcommand{\dashto}{\dashrightarrow}

\DeclareMathOperator{\codim}{codim}
\DeclareMathOperator{\im}{im}

\title{\bf Lie Groupoids and Differentiable Stacks}

\author{Matias L. del Hoyo\thanks{Funded by FCT Postdoc Fellowship PTDC/MAT/117762/2010, Portugal.}}

\date{}

\def\address#1{\noindent {\sc #1}}
\def\email#1{\noindent \sf{#1}}
\def\url#1{\noindent \sf{#1}}

\begin{document}

\maketitle

\vspace{-10pt}
\centerline{\it Instituto Superior T\'ecnico, Lisboa, Portugal.}

\vspace{20pt}

\noindent{\bf Abstract.}{
This is a concise introduction to the theory of Lie groupoids, with emphasis in their role as models for stacks. After some preliminaries, we review the foundations on Lie groupoids, and we carefully study equivalences and proper groupoids.

Differentiable stacks are geometric objects which have manifolds and orbifolds as special instances, and can be presented as the transverse geometry of a Lie groupoid. %They can alternatively be defined using stacks.
Two Lie groupoids are equivalent if they are presenting the same stack, and proper groupoids are presentations of separated stacks, which by the linearization theorem are locally modeled by linear actions of compact groups.
We discuss all these notions in detail.

Our treatment diverges from the expositions already in the literature, looking for a complementary insight over this rich theory that is still in development.}

\medskip

\noindent{\bf Mathematics Subject Classification (2000).}{
22A22, % Topological groupoids (including differentiable and Lie groupoids)
 58H05. % Pseudogroups and differentiable groupoids
}

\medskip

\noindent{\bf Keywords.}{
Lie Groupoids, Differentiable Stacks, Linearization.
}

\medskip

\begin{small}

\tableofcontents

\end{small}

%%%%%%%%%%%%%%%%%%%%%%%%%%%%%%%%%%

\section{Introduction}

% Lie groupoids
Lie groupoids constitute a general framework which has received much attention lately.
They generalize group actions, submersions, foliations, pseudogroups and principal bundles, among other construction, providing a new perspective to classic geometric questions and results.
Besides, Lie groupoids can be seen as an intermediate step in defining differentiable stacks, some geometric objects admiting singularities and generalizing both manifolds and orbifolds.

% Orbispaces
Differentiable stacks can be defined within the language of stacks, an abstract concept introduced by Grothendieck in his work on algebraic geometry (cf. \ref{stack}).
Here we avoid that paraphernalia and follow a more concrete approach, according to which a differentiable stack is what encodes the transversal geometry of a Lie groupoid.

% Equivalences
Every Lie groupoid has an underlying differentiable stack, and two Lie groupoids have the same one if they are equivalent.
These equivalences, sometimes called Morita equivalences, can be realized either by principal bibundles or by chains of fully faithful essentially surjective maps.
Many properties of Lie groupoids are invariant under equivalences because they are actually properties of their orbit stacks.
%A fundamental problem is to understand and classify Lie groupoids up to equivalences.

% Proper groupoids
Proper groupoids constitute an important family of Lie groupoids. It includes manifolds, compact groups, submersions and proper actions, among others. They have a Hausdorff orbit space and their isotropy groups are compact. Moreover, they can be linearized around an orbit, which implies that their underlying stacks, called separated, can locally be modeled by linear actions of compact groups.

\medskip

% About this work
These notes contain the foundations of Lie groupoids with special emphasis in equivalences and proper groupoids, including the relatively new results on linearization. 
We pursued a self-contained presentation with a stress in examples, and avoiding when possible technical digressions.

Even though most of the material is already available in the literature, we provide a new perspective over known results, such as new proofs and examples, and we also include new subsidiary results:
a description on the differential of the anchor (cf. \ref{diffanch}), 
a characterization of Morita maps by means of the normal representations (cf. \ref{charequi}), a discussion of stable orbits (cf. \ref{stable2}),
and a reduction on Zung's theorem that simplifies the proof (cf. \ref{step1}),
to mention some of them.

Our manuscript omits Lie algebroids, the infinitesimal counterpart of Lie groupoids, and the interesting theory they play together. This can be found elsewhere, see e.g. \cite{cw}, \cite{cf}, \cite{dz}, \cite{mkbook}, \cite{mm}. 
Note that the topics studied here, equivalences and proper groupoids, have not a known infinitesimal version in the interplay between groupoids and algebroids.

To finish, let us mention that the linearization problem is studied here in the spirit of the original works \cite{weinstein} and \cite{zung}, and the more recent paper \cite{cs}. 
A completely new approach will be presented in \cite{riemannian}, where 
we define compatible metrics with the groupoid structure, and establish the linearization by exponential maps, providing both a stronger version and a simpler proof for this important theorem.

% In section 1 we recall some preliminaries in proper maps, pullbacks and quotients of manifolds, and the structure of submersions. 
% Section 2 presents Lie groupoids starting with a definition, the fundamental examples, and the basic facts regarding their structure, following with groupoid actions and linear representations, especially the normal representation, and ending with an analysis of the differential of the anchor map. During section 3 we define weak equivalences as fully faithful essentially surjective maps, characterize them by using the normal representations, we use homotopy pullbacks to establish equivalences and generalized maps, and then we relate this approach with the other common one involving principal bibundles.
% The last section 4 deals with proper groupoids, their basic properties, a discussion on stability and the averaging arguments, and it finishes with linearization: first Zung's theorem as a local description for singular spaces, and then the general case.

\medskip

\noindent{\it Organization. }
In section 2 we recall some preliminaries, in section 3 we present a brief self-contained introduction to Lie groupoids, we carefully discuss equivalences of Lie groupoids in section 4, and finally in section 5 we deal with proper groupoids and linearization. A more detailed description can be found at the beginning of each section.

\medskip

\noindent{\it Acknowledgments. }
These notes were born out of expositions in workgroup seminars at IMPA, Rio de Janeiro, and at IST, Lisbon.
I thank Rui Loja Fernandes for encouraging me to write them and for his comments on preliminary versions.
I am indebted to him and to Henrique Bursztyn for their guidance and support.
%I thank CNPq-Brasil and FCT-Portugal for the financial support.
I also want to thank Fernando Cukierman, Reimundo Heluani, Alejandro Cabrera, Thiago Drummond, Olivier Brahic, David Martinez and Daniele Sepe for the fruitful conversations.

%%%%%%%%%%%%%%%%%%%%%%%%%%%%%%%%%%
%%%%%%%%%%%%%%%%%%%%%%%%%%%%%%%%%%

\section{Preliminaries}

Throughout this section we collect basic facts that are scattered on the literature and provide alternative formulations for some of them.
The topics are proper maps, pullbacks and quotients of smooth manifolds, and the structure of submersions. These results will be needed in the subsequent sections.

We refer to \cite{bourbaki} for a detailed exposition on proper maps, and to \cite{dieudonne} and \cite{lang} for generalities on differential geometry. 

%%%%%%%%%%%%%%%%%%%%%%%%%%%%%%%%%%

\subsection{Properness and properness at a point}\label{sectionpropmaps}

All our spaces are assumed to be second countable, locally compact and Hausdorff. This includes smooth manifolds and Hausdorff quotients of them.

Let $X,Y$ be two spaces. A continuous map $f:X\to Y$ is {\bf proper} if it satisfies any, and hence all, of the following equivalent conditions:
%\begin{proposition}\label{defiprop}
%The following are equivalent:
\begin{itemize}
 \item for all map $Z\to Y$ the base-change
 $\tilde f:X\times_Y Z\to Z$ of $f$ is closed;
$$\xymatrix{X\times_{Y} Z \ar[r] \ar[d]_{\tilde f} \ar@{}[dr]|(.7){\lefthalfcap}
& X \ar[d]^{f} \\ Z\ar[r] & Y}$$
 \item $f$ is closed and has compact fibers;
 \item $f^{-1}(K)\subset X$ is compact for all $K\subset Y$ compact; and
 \item every sequence $(x_n)\subset X$ with $f(x_n)\to y$ admits a convergent subsequence $x_{n_k}\to x$.
\end{itemize}
%\end{proposition}
 
The proofs of the equivalences are rather standard, see e.g. \cite{bourbaki}. Let us remark that the first two formulations remain equivalent when working with any topological spaces, and they are equivalent to the other two only under our hypothesis.

\begin{example}
An inclusion $A\subset X$ is proper if and only if $A$ is a closed subspace of $X$. 
A projection $F\times X\to X$ is proper if and only if the fiber $F$ is compact.
\end{example}

Proper maps form a {\bf nice class} of maps, namely (1) every homeomorphism is proper, (2) composition of proper maps is proper, and (3) the base-change of a proper map is proper.

% A product map $f\times f':X\times X'\to Y\times Y'$ is proper if and only if each of the factors is so. This is easy to see, for instance, by using sequences.
% 
\bigskip

% properness at a point
The notion of properness admits a punctual version. The map $f:X\to Y$ is {\bf proper at $y$} if any sequence $(x_n)\subset X$ such that $f(x_n)\to y$ has a convergent subsequence $x_{n_k}\to x$ (cf. \cite{dk}). Clearly $f$ is proper if and only if it is proper at every point of $Y$. % Also, given $f:X\to Y$ proper at $y$ and $g:Z\to Y$, the base-change of $f$ along $g$ is proper at $z$ for any $z$ such that $g(z)=y$.
It turns out that properness is an open condition, namely the points at which $f$ is proper is an open of $Y$.

\begin{proposition}\label{openproper}
If $f:X\to Y$ is proper at $y$ then there exists an open $y\in V$ such that $f|_V:f^{-1}(V)\to V$ is proper.
\end{proposition}

\begin{proof}
Take a sequence of open subsets $(V_n)\subset Y$ such that $V_n\searrow y$, and a sequence of compact subsets $(K_n)\subset X$ such that ${\rm int}\, K_n\nearrow X$. If for some $n$ we have $f^{-1}(V_n)\subset K_n$ then $f|_{V_n}$ satisfies that preimages of compact sets are compact and hence is proper. If for all $n$ we can take $x_n\in f^{-1}(V_n)\setminus K_n$ then $f(x_n)\to y$ and $(x_n)$ has no convergent subsequence, which contradicts the hypothesis.
\end{proof}

% the tube principle
To better understand the notion of properness at a point let us introduce the following natural definition. Given $f:X\to Y$, $y\in Y$, $F=f^{-1}(y)$, we say that $f$ satisfies the {\bf tube principle} at $y$ if for every open $U$, $F\subset U\subset X$, there exists an open $V$, $y\in V\subset Y$, such that $f^{-1}(V)\subset U$. In other words, any open containing the fiber must also contain an open tube around it. % We have the following characterization of properness at a point.

\begin{proposition}\label{charpatx}
A map $f:X\to Y$ is proper at $y$ if and only if the fiber $F$ is compact and $f$ satisfies the tube principle at $y$.
\end{proposition}

\begin{proof}
If $f:X\to Y$ is proper at $y$ then we have seen that there is a tube on which $f$ is proper. Then we can assume that $f$ is proper, and therefore closed. Now, if $U$ is an open containing the fiber $F$, we can take $V=Y\setminus(f(X\setminus U))$. 

Conversely, suppose that $(x_n)$ is such that $f(x_n)\to y$ but $(x_n)$ has no convergent subsequence. If $F$ is compact, then $x_n$ will belong to $F$ at most finitely many times.
Dropping the first terms of the sequence we can assume that $x_n\notin F$ for any $n$, and then $U=X\setminus\{x_n\}$ is an open around $F$ which does not contain any tube.
\end{proof}

\begin{example}
Next examples show the necessity of the two conditions.
The projection $S^1\setminus\{i\} \to \R$, $\exp(it)\mapsto\cos(t)$, has compact fiber at $0$ but it does not satisfy the tube principle. The smooth map $\phi:\R\to\R$, $\phi(x)=0$ for $x\leq 0$, $\phi(x)=\exp(-1/x)$ for $x>0$, satisfies the tube principle at $0$ but its fiber is not compact.
\end{example}

%%%%%%%%%%%%%%%%%%%%%%%%%%%%%%%%%%

\subsection{Good pullbacks of manifolds}\label{goodpull}

% good pullbacks of manifolds
The pullback of two maps in the category of smooth manifolds, if it exists, may behave badly with respect to the underlying topologies and also to the construction of tangent spaces.
Let us illustrate this with examples.

\begin{example}
The next square is a pullback of manifolds, but the induced diagram between the tangent spaces at $0$ is not a pullback.
$$\begin{matrix}\xymatrix{0 \ar[r] \ar[d] \ar@{}[dr]|(.7){\lefthalfcap}& \R \ar[d]^g \\ \R\ar[r]^f & \R^2}\end{matrix}
\qquad \qquad \begin{matrix}
        f(t)=(t,0) \\ g(t)=(t,t^2)
       \end{matrix}$$
The intersection between the two curves is something more than the point, it contains some extra infinitesimal data.
\end{example}

\begin{example}
Let $\alpha$ be irrational and let $D\subset\R\times\R$ be the set-theoretic pullback 
$$\begin{matrix}
\xymatrix{D \ar[r] \ar[d] \ar@{}[dr]|(.7){\lefthalfcap}& \R \ar[d]^g \\ \R\ar[r]^f & S^1\times S^1}\end{matrix}
\qquad \qquad \begin{matrix}
        f(t)=(e^{i\alpha t},e^{it}) \\ g(t)=(e^{it},e^{i\alpha t})
       \end{matrix}$$
viewed as a discrete manifold. The square is not a topological pullback, for the intersection of the two dense curves on the torus has a non-trivial topology. However, it is a pullback of manifolds, for a smooth map $M\to S^1\times S^1$ whose image lies in the intersection of the two curves has to be constant.
\end{example}

We say that a pullback of smooth manifolds is a {\bf good pullback} if:
$$\xymatrix{M_1\times_{M} M_2 \ar[r]^{\tilde f_1} \ar[d]_{\tilde f_2} \ar@{}[dr]|(.7){\lefthalfcap}
& M_2 \ar[d]^{f_2} \\ M_1\ar[r]_{f_1} & M}$$
\begin{enumerate}
\item It is a pullback of the underlying topological spaces, and
\item It induces pullbacks between the tangent spaces, say for each $x_1,x_2,x$ such that $f_1(x_1)=x=f_2(x_2)$ the following sequence is exact.
$$0 \to T_{(x_1,x_2)}(M_1\times_M M_2) \to T_{x_1}M_1\times T_{x_2}M_2 \to T_xM$$
The last arrow is given by $(v,w)\mapsto d_{x_1}f_1(v)-d_{x_2}f_2(w)$.
\end{enumerate}
In other words, the pullback is good if the map $M_1\times_M M_2 \to M_1\times M_2$ is a closed embedding with the expected tangent space.
%It follows that if $M\times_N M'$ is a good pullback then their tangent bundles also define a good pullback, say $T(M\times_N M')\cong TM\times_{TN}TM'$.

The standard criterion for the existence of pullbacks is by means of transversality. Recall that two smooth maps $f_1:M_1\to M$, $f_2:M_2\to M$ are {\bf transverse} if $d_{x_1}f_1(T_{x_1}M_1)+d_{x_2}f_2(T_{x_2}M_2)=T_xM$ for all $x_1,x_2,x$ making sense.

\begin{proposition}\label{pullback of manifolds}
If $f_1:M_1\to M$ and $f_2:M_2\to M$ are transverse, then their pullback $M_1\times_M M_2$ exists and it is a good pullback.
\end{proposition}

For a proof see e.g. \cite{lang}.

% 
% Note that when the maps are transverse the sequence of tangent spaces mentioned in (ii) is also exact at the end, namely the last map is onto.
%$$0 \to T_{(x_1,x_2)}(M_1\times_M M_2) \to T_{x_1}M_1\times T_{x_2}M_2 \to T_xM \to 0$$
% \begin{proof}
% When $f'$ is the inclusion of a closed embedded submanifold $M'\subset N$, then $M'$ is locally defined as the fiber of a submersion $\pi$, and the transversality condition insures that the composition $\pi f$ is also a submersion, hence its fiber is an embedded submanifold of $M$ with the right tangent space, which gives the local structure of $M\times_N M'\subset M$. 
% 
% The general case follows from the previous one, for if one of the following squares is a pullback then so does the other, and $f\transverse f'$ iff $f\times f' \transverse \Delta$. 
% $$\xymatrix{M\times_N M' \ar[r] \ar[d] & M' \ar[d]^{f'} \\ M\ar[r]_{f} & N}
% \qquad 
% \xymatrix{M\times_N M' \ar[r] \ar[d] & M\times M'\ar[d]^{f\times f'} \\ N \ar[r]_{\Delta} & N\times N}$$
% See e.g. \cite{lang} for further details.
% \end{proof}

\begin{remark}\label{pullbacks&submersions}
A submersion is transverse to any other map, thus the pullback between a submersion and any other map always exists and it is good.

The base-change of a submersion is always a submersion. As a partial converse, with the notations above, if $\tilde f_1$ is a submersion and $f_2$ is a surjective submersion, then $f_1$ has to be a submersion as well, as it follows from the induced squares of tangent vector spaces.
\end{remark}

%%%%%%%%%%%%%%%%%%%%%%%%%%%%%%%%%%

\subsection{Quotients of manifolds}

% introduction
Given $M$ a smooth manifold and $R\subset M\times M$ an equivalence relation on it, it is natural to ask whether if the quotient set $M/R$ can be regarded as a new manifold.
Of course, this is not the case in general.

\begin{example}
The action by rotations on the plane $S^1\action\R^2$ leads to a quotient $\R^2/S^1$ which is not a manifold for it is not locally euclidean. We do not allow manifolds to have border.
\end{example}

\begin{example}
The foliation $F$ on $\R^2\setminus\{(0,0)\}$ by horizontal lines defines an equivalence relation on which the quotient $(\R^2\setminus\{(0,0)\})/F$ is locally euclidean, but it is not Hausdorff.
\end{example}

% necessary conditions
Given $M$ and $R$, if the quotient $M/R$ admits a manifold structure and the projection $M\to M/R$ is a submersion, then the following square happens to be a good pullback.
$$\xymatrix{R \ar[r] \ar[d] \ar@{}[dr]|(.7){\lefthalfcap} & M \ar[d] \\ M\ar[r] & M/R}$$
Then $R\subset M\times M$ has to be a closed embedded submanifold and the projections $\pi_1|_R,\pi_2|_R:R\to M$ have to be submersions.
It turns out that these conditions are sufficient to define a manifold structure on the quotient.

% godement criterion
\begin{proposition}[Godement criterion]\label{godement}
If an equivalence relation $R\subset M\times M$ is a closed embedded submanifold and $\pi_1|_R,\pi_2|_R:R\to M$ are submersions, then $M/R$ inherits a unique manifold structure that makes the projection $M\to M/R$ a submersion.
\end{proposition}

Note that if one of the projections is a submersion then so is the other. For the construction of such a manifold structure on $M/R$ we refer to \cite{licogeom}, see also \cite{serrelie}.

Given $M$ and $R\subset M\times M$ as in \ref{godement}, we can identify the smooth maps $M/R\to Z$ with those maps $M\to Z$ which are constant over the classes defined by $R$.
$$C^\infty(M/R,Z)\cong C_R^\infty(M,Z)\subset C^\infty(M,Z)$$
In fact, since $M\to M/R$ is a surjective submersion then (1) it is open and hence a topological quotient, and (2) it admits local sections, hence a continuous map $M/R\to Z$ is smooth if and only if the composition $M\to M/R\to Z$ is so.
This proves in particular the uniqueness in \ref{godement}.

\bigskip

The following corollary is probably better known than the criterion itself.
Let $G$ be a Lie group and consider an action $G\action M$, $(g,x)\mapsto g\cdot x$, over a manifold $M$. Recall that the action is {\bf proper} if the map $G\times M\to M\times M$, $(g,x)\mapsto(g\cdot x,x)$ is proper. 

\begin{corollary}\label{quotgrou}
If $G\action M$ is a free proper action of a Lie group on a manifold, then the quotient $M/G$ inherits a unique manifold structure that makes the projection $M\to M/G$ a submersion.
\end{corollary}

\begin{proof}[Sketch of proof]
Given $x\in M$, consider the map $G \to M$, $g\mapsto g.x$.
If $v\in T_1G $ is a nonzero vector in the kernel of its differential, then the 1-parameter group it generates is included in the isotropy of $x$. This proves that there is no such a $v$.
The same argument shows that $G \times M\to M\times M$, $(g,x)\mapsto (gx,x)$ is an injective immersion. Since it is also proper it turns out to be a closed embedding. The composition $G \times M \to M\times M \xto{\pi_2} M$ is clearly a submersion and we can apply Godement criterion \ref{godement}. \end{proof}

For alternative approaches see \cite{dk} and \cite{sharpe}.

%%%%%%%%%%%%%%%%%%%%%%%%%%%%%%%%%%

\subsection{The structure of submersions}

% structure theorem for submersions
The constant rank theorem implies that, in a neighborhood of a point, a submersion looks as a projection . When moving along the fiber this leads to the following description of the structure of submersions.

\begin{proposition}\label{strusubm}
Let $f:M\to N$ be a submersion, $y\in N$, $F=f^{-1}(y)$.
There are opens $U\supset F$ and $V\ni y$, and an open embedding $i_U:U\to F\times V$ extending the obvious inclusion $F\to F\times y$ and satisfying $f|_U=\pr \circ i_U$.
$$\begin{matrix}
M & \supset & U & \xto{i_U} & F\times V \\
\downarrow & & \downarrow & & \downarrow\\
N & \supset & V & = & V
\end{matrix}$$
\end{proposition}

Note that this is a local statement around $y$, thus we may change $V$ by any smaller neighborhood. In particular we may take $V\cong\R^n$ a ball-like open and this way compare $f$ with the projection $F\times\R^n\to\R^n$.

Proposition \ref{strusubm} admits rather elementary proofs. We propose the following using Riemannian geometry, maybe more sophisticated, but with interesting generalizations (see the proof of \ref{linearization}, and also \cite{riemannian}).

\begin{proof}[Sketch of proof]
% We can assume $N=\R^n$. For $n=1$, we can lift the vector field $\partial_t$ on $\R$ to a vector field $X$ on $M$ by using a partition of unity. Then we can obtain $i_U$ as the inverse of the restriction of the flow $\phi_X:M\times\R\to M$ of $X$ to a suitable neighborhood of $F\times 0$. The case $n>1$ follows by an inductive argument.
% 
% Other proof, maybe more sophisticated, is by considering a tubular neighborhood $F\subset T\subset M$ with projection $\pi:T\to F$. Then the map $(\pi,f):T\to F\times \R^n$ is compatible with the projections, restricts to a diffeomorphism on $F$, and is etal over $F$. It follows by a standard metric argument that $(\pi,f)$ is still injective in a neighborhood of $F$, hence an open embedding.
We can assume $N=\R^n$.
Endow $M$ and $N$ with Riemannian metrics for which $f$ is a Riemannian submersion, that is, such that $df_x|_{({\ker df_x})^\bot}$ is an isometry for all $x$.
Such metrics can be easily constructed.

The normal bundle $NF$ of the fiber $F$ is trivial. We identify it with the vectors orthogonal to $TF$. The geodesics associated to these vectors are preserved by $f$, thus the exponential maps of the metrics yield a commutative diagram
$$\begin{matrix}F\times\R^n & \cong & NF & \xto{\exp} & M \\
\downarrow & & \downarrow & & \downarrow \\
\R^n & \cong & T_yN & \xto{\exp} & N\end{matrix}$$
Since the map $F\times\R^n\cong NF\xto\exp M$ is injective over $F\times 0$ and its differential is invertible over $F\times 0$, it follows from a standard metric argument that it is still injective in an open around $F\times 0$, hence an open embedding, and we can take $i_U$ as its inverse.
\end{proof}

% Ehresmann theorem
As a straightforward application we obtain the well-known Ehresmann Theorem.
%We give here a strengthened version.

\begin{proposition}[Ehresmann Theorem]\label{etheorem}
Let $f:M\to N$ be a submersion, $y\in N$, $F=f^{-1}(y)$. The following are equivalent:
\begin{enumerate}
\item $f$ is locally trivial at $y$ and $F$ is compact;
\item $f$ is proper at $y$; and
\item $f$ satisfies the tube principle at $y$.
\end{enumerate}
\end{proposition}
\begin{proof}
The implications (1) $\then$ (2) $\then$ (3) are obvious (cf. \ref{charpatx}).

To prove (3) $\then$ (2) we need to show that the fiber is compact (cf. \ref{charpatx}). Because of \ref{strusubm} it is enough to study the case $U\subset F\times \R^n\to \R^n$. Given $(x_n)\subset F$, we will see that it admits a convergent subsequence. If not, we can take $v_n\in\R^n$ such that $0<||v_n||<1/n$ and $(x_n,v_n)\in U$, then $U\setminus\{(x_n,v_n)\}_n$ is an open around $F$ not containing any tube, which contradicts (3). 

Finally, assume (2) and let us prove (1). Since $f$ is proper at $y$, the fiber $F$ is compact, and $f$ satisfies the tube principle at $y$. Then, in \ref{strusubm}, by shrinking $U$, we can assume it is saturated. Since the projection $F\times V\to V$ is also proper we can shrink $U$ again so as to make $i_U(U)$ saturated, and $U$ will remain saturated as well, proving local triviality.
\end{proof}

% a further version
Ehresmann theorem admits an interesting particular case on which the hypothesis rely only on the topology of the fibers.

\begin{corollary}\label{ehreconn}
Let $f:M\to N$ be a submersion, $y\in N$, $F=f^{-1}(y)$. If $F$ is compact and the nearby fibers are connected then $f$ is proper at $y$.
\end{corollary}

\begin{proof}
Let $U,V$ and $i_U:U\to F\times V$ be as in \ref{strusubm}.
Since $F$ is compact the projection $F\times V\to V$ is proper, and by the tube principle we may shrink $U$ so as to make $i_U(U)$ saturated. We may suppose that $U$ intersects only connected fibers $F'$. For each of these fibers $F'$ we have that $U\cap F'\cong i_U(U\cap F')\cong F$ is compact, hence closed and open on $F'$. This proves $F'\subset U$ and that $U$ is saturated as well.
\end{proof}

%%%%%%%%%%%%%%%%%%%%%%%%%%%%%%%%%%
%%%%%%%%%%%%%%%%%%%%%%%%%%%%%%%%%%

\section{Lie groupoids}

This section starts with definitions and examples on Lie groupoids.
Then we discuss groupoid actions and linear representations, with special emphasis in the normal representation, which encodes the linear infinitesimal information around an orbit. We describe then the differential of the anchor map, and use this to characterize two important families: submersion groupoids and transitive groupoids.
Finally we discuss principal groupoid-bundles.

We suggest \cite{cw}, \cite{cf}, \cite{dz}, \cite{mkbook}, \cite{mm} as standard references for this section.

%%%%%%%%%%%%%%%%%%%%%%%%%%%%%%%%%%

\subsection{Definitions and basic facts}

% graphs
A {\bf smooth graph} $G\toto M$ consists of a manifold $M$ of objects, a manifold $G$ of arrows, and two submersions $s,t:G\to M$ indicating the {\bf source} and {\bf target} of an arrow.
We often write the arrows from right to left, thus by $y\xfrom g x$ we mean $x,y\in M$, $g\in G$, $s(g)=x$ and $t(g)=y$.
We call $M$ the {\bf base} of the graph. 

% groupoids
A {\bf Lie groupoid} consists of a smooth graph $G\toto M$ endowed with a smooth associative {\bf multiplication} $m$,
$$m:G\times_M G\to G \quad (z \xfrom{g_2}y,y\xfrom{g_1}x)\mapsto (z \xfrom{g_2g_1}x)$$
 where $G\times_M G=\{(g_2,g_1)|s(g_2)=t(g_1)\}\subset G\times G$ is the submanifold of composable arrows. This multiplication is required to have a {\bf unit} $u$ and an {\bf inverse} $i$, which are smooth maps
$$u:M\to G \quad x\mapsto (x\xfrom{1_x} x) \qquad i: G\to G \quad (y\xfrom g x)\mapsto (x\xfrom{g^{-1}}y)$$
 satisfying the usual axioms
$gg^{-1}=1_y$, $g^{-1}g=1_x$, $g1_x=g$ and $1_yg=g$ for all $y\xfrom g x$.
%The whole structure of the Lie groupoid $G$ can be encoded in the two manifolds $G,M$ and the five 
We refer to $s,t,m,u,i$ as the {\bf structural maps} of the Lie groupoid.
By an abuse of notation, we denote $(G\toto M, m)$ just by $G\toto M$, or even  $G$.

\bigskip

% Given $G\toto M$ a smooth graph and $A,B\subset M$, we write $G(B,A)=s^{-1}(A)\cap t^{-1}(B)$ for the set of arrows from $A$ to $B$. The subset $G(B,A)\subset G$ may fail to be submanifolds even when $A$ and $B$ are so, due to transversality conditions.
% 
% Given $G\toto M$ a smooth graph and $A,B\subset M$, we use the notations
% $G(B,A)=s^{-1}(A)\cap t^{-1}(B)$ for the arrows from $A$ to $B$, and $G_A=G(A,A)$.
% The subset $G(B,A)\subset G$ may fail to be submanifolds even when $A$ and $B$ are so, due to transversality conditions.

Given $G\toto M$ a smooth graph, a {\bf bisection} $B\subset G$ is an embedded submanifold such that the restrictions $s_B,t_B:B\to M$ of the source and target are open embeddings. Naming $U=s(B)$ and $V=t(B)$, the maps $s_B:B\to U$ and $t_B:B\to V$ are diffeomorphisms, and the composition $f_B=t_B s_B^{-1}:U\to V$ is called the {\bf underlying map} to $B$. We can visualize $B$ as a bunch of arrows from $U$ to $V$. 
Given any $g\in G$, it is easy to see that there always exists a bisection $B$ containing $g$.

When $G\toto M$ is a Lie groupoid, the bisections can be composed, inverted, and every open $U\subset M$ has a unitary bisection. The underlying maps to bisections of $G$ define hence a pseudogroup on $M$, the {\bf characteristic pseudogroup} of $G\toto M$.
There are local {\bf translation} and {\bf conjugation} maps associated to a bisection $B$.
$$L_B:G(U,-)\to G(V,-) \quad (y\xfrom h x )\mapsto (f_B(y)\xfrom{s_B^{-1}(y) h} x)$$
$$C_B:(G_U\toto U)\to (G_V\toto V) \quad (y\xfrom h x)\mapsto (f_B(y)\xfrom{s_B^{-1}(y) h s_B^{-1}(x)^{-1}}f_B(x))$$
Here we are using the notations $G(A,-)=t^{-1}(A)$ and $G_A=s^{-1}(A)\cap t^{-1}(A)$.

\bigskip

Given $G \toto M$ a Lie groupoid and $x\in M$, the {\bf s-fiber} at $x$ is defined by $G(-,x)=s^{-1}(x)\subset G$, the {\bf isotropy group} at $x$ is $G_x=s^{-1}(x)\cap t^{-1}(x)\subset G$, and the {\bf orbit} of $x$ is the set $O_x=\{y|\exists y\xfrom g x\}=t(G(-,x))$.

\begin{proposition}\label{orbits&isotropy}
Given $G\toto M$ and $x,y\in M$, the subset $G(y,x)\subset G$ is an embedded submanifold. In particular $G_x$ is a Lie group. The orbit $O_x\subset M$ is a (maybe not embedded) submanifold in a canonical way.
\end{proposition}

\begin{proof}
Denote by $t_x:G(-,x)\to M$ the restriction of the target map to the $s$-fiber.
Given $g,g'\in G(-,x)$ and a bisection $B$ containing $g'g^{-1}$, we can locally write $t_xL_B=f_Bt_x$ in a neighborhood of $g$, and since $L_B$ and $f_B$ are invertible, the rank of $t_x$ at the points $g$ and $L_B(g)=g'$ agree. This shows that $t_x$ has constant rank, hence its fibers $G(y,x)$ are embedded submanifolds. In particular $G_x$, with the operations induced by those of $G$, becomes a Lie group. This group acts freely and properly on $G(-,x)$ by the formula
$$G(-,x)\times G_x\to G(-,x) \qquad (y\xfrom{g'} x, x\xfrom g x) \mapsto (y \xfrom{g'g} x)$$
We can identify the quotient $G(-,x)/G_x$ with the orbit $O_x\subset M$, and regard it as a submanifold in a canonical way (cf. \ref{quotgrou}).
\end{proof}

% characteristic foliation
The {\bf orbit space} $M/G$ is the set of orbits with the quotient topology. The quotient map $q:M\to M/G$ is open, as it follows from a simple argument on bisections.
The partition of $M$ into the connected components of the orbits is a singular foliation, called the {\bf characteristic foliation} of $G$.

The space $M/G$ is not a smooth manifold in general. We can think of the Lie groupoid $G\toto M$ as a way to describe certain smooth singular data on it (cf. \ref{sectionss}).
 % The groupoid $G$ is called {\bf regular} if this is a regular foliation, that is, if all the orbits have the same dimension.

\bigskip

A {\bf map of graphs} $\phi:(G\toto M)\to (G'\toto M')$ is a pair of smooth maps
$$\phi^{ar}:G\to G' \qquad \phi^{ob}:M\to M'$$
that preserves the source and the target. We usually denote $\phi^{ar}$ and $\phi^{ob}$ simply by $\phi$.
A {\bf map of Lie groupoids} is a map between the underlying graphs that also preserves the multiplicative structure, that is, commutes with $m$, and therefore with $u$ and $i$.
We denote the category of Lie groupoids and maps between them by 
$$\{\text{Lie Groupoids}\}$$
Out of a Lie groupoid $G\toto M$ we have constructed a family of Lie groups $\{G_x\}_{x\in M}$ and a quotient map $M\to M/G$. These constructions are functorial, a map $\phi:(G\toto M)\to(G'\toto M')$ induces Lie group homomorphisms $\phi_*:G_x\to G_{\phi(x)}$ and a continuous map between the orbit spaces $\phi_*:M/G\to M'/G'$.

% Note that for $b\in B$ we have 
% $$T_gB\oplus \ker d_gs = T_gG = T_gB\oplus \ker d_gt$$

% \begin{proof}
% The subspaces $\ker_g ds, \ker _g dt\subset T_gG$ have the same dimension, hence they admit a common complement, and we can take $B\subset G$ to be a small embedded submanifold such that $T_gB$ complements both subspaces.
% \end{proof}

% bisection on a Lie groupoid
% Given $G\toto M$ a Lie groupoid and $B$ a bisection, call $\sigma=(s|_B)^{-1}$.
% The  left {\bf translation}
% is a diffeomorphism satisfying $s\circ L_B= s$ and $t\circ L_B = \phi_B \circ t$, and the right translation is analog.
% The {\bf conjugation} 
%Here we use the notations $G(-,A)=t^{-1}(A)$ and $G_A=s^{-1}(A)\cap t^{-1}(A)$. Note that since $U\subset M$ is open then there is an induced Lie groupoid structure $G_U\toto U$. 
%Thus, the groupoid looks exactly the same around points that are linked by an arrow.

% The whole structure of the Lie groupoid $G$ can be encoded in the two manifolds $G,M$ and the five smooth maps $s,t,m,u,i$ as shown in the diagram
% $$\xymatrix{
% G \times_M G \ar[r]^(.6)m &
% G \ar@<.5ex>[r]^t \ar@<-.5ex>[r]_s  \ar@(ul,ur)[];[]_i&
% M \ar@/^2.5ex/[l]^u}$$

% \begin{remark}
% Given two graphs $G\toto M$ and $G'\toto M$, we can consider the pullback $G\times_M G'$ between $s$ and $t'$, and regard it as a graph $G\times_M G'\toto M$ with source $s'\circ\pi_2$ and target $t\circ \pi_1$.
% In a fancy language, a Lie groupoid is a group object in the monoidal category of smooth graphs over $M$ together with this product.
% \end{remark}

%%%%%%%%%%% EXAMPLES

\subsection{Some examples}\label{sectionexamples}

Lie groupoids constitute a common framework to work with several geometric structures.
We give here some of the fundamental examples.
We do not include the examples of foliations and pseudogroups, for in these cases, the manifold $G$ may not be second countable nor Hausdorff, as we require. These and other important examples can be found in \cite{mkbook} and \cite{mm}. 

% Manifolds & Lie groups
\begin{example}[Manifolds and Lie groups]
A manifold $M$ gives rise to the {\bf unit groupoid} with only unit arrows, where the five structural maps are identities, the isotropy is trivial and the orbits are just the points.
As other extremal case, a Lie group $G$ can be seen as a Lie groupoid with a single object.
$$M \quad \leadsto \quad M\toto M
\qquad \qquad
G  \quad \leadsto \quad G \toto \ast$$ 
These constructions preserve maps. We will identify manifolds and Lie groups with their associated Lie groupoids.
\end{example}

% Actions
\begin{example}[Group actions]\label{grouacti}
If $G \action M$ is a Lie group acting over a manifold, the {\bf action groupoid} $G \ltimes M = (G \times M\toto M)$ is defined with source the projection, target the action, 
and multiplication, unit and inverse maps induced by those of $G$.
$$G \action M \quad \leadsto \quad G \times M\toto M$$
A typical arrow in the action groupoid has the form  $g\cdot x \xfrom{(g,x)} x$.
Orbits and isotropy correspond to the usual notions for actions. 
A map of actions induce a map between the corresponding action groupoids, but in general there are more maps than these. 

Action groupoids $G\ltimes M$ are fundamental examples. Not every Lie groupoid is of this form, but  we will see that every {\it proper} groupoid is locally {\it equivalent} to one of these (cf. \ref{linearization}).
%In fact, many results on Lie groupoids are extensions of previous results on equivariant geometry.
%Nevertheless, not every Lie groupoid is isomorphic to an action groupoid, as we can see in some of the following examples.
\end{example}

% Submersions
\begin{example}[Submersions]\label{examsubm}
A submersion $q:M\to N$ yields a {\bf submersion groupoid} $M\times_N M \toto M$ with one arrow between two points if they belong to the same fiber. 
$$ M\to N \quad \leadsto \quad M\times_N M \toto M$$
%Given a submersion $q:M\to N$, since every submersion is locally a projection, the groupoid $\G(q)$ locally looks as a product of a unit groupoid and a pair groupoid: for every $x\in M$ there exists an open $x\in U\subset M$ and an isomorphism $\G(q)_U\cong (\underline{V\times V})\times \underline W$.
The isotropy of a submersion groupoid is trivial and the orbits are the fibers of $q$.
A map between submersion groupoids $(M\times_N M\toto M)\to(M'\times_{N'}M'\toto M')$ is the same as a commutative square of smooth maps. We will characterize later the groupoids arising from submersions (cf. \ref{anchinje}).
% $$ (M\times_N M\toto M)\to(M'\times_{N'}M'\toto M')
% \qquad \leftrightarrow \qquad
% \begin{matrix} M & \to & M' \\ \downarrow & & \downarrow \\ N & \to & N'\end{matrix}$$

Given a manifold $M$, the submersion groupoid of the identity $\id_M:M\to M$ yields the unit groupoid, and the submersion groupoid of the projection $\pi_M:M\to\ast$ is the {\bf pair groupoid} $M\times M\toto M$, which has exactly one arrow between any two objects.
Other interesting case arise from an open cover $\U=\{U_i\}_i$ of $M$. The several inclusions $U_i\to M$ yield a surjective submersion $\coprod_i U_i\to M$, and this yields a {\bf covering groupoid} $\coprod_{i,j}U_i\cap U_j \toto \coprod_i U_i$, also called {\bf Cech groupoid}.

%$$G(\id_M)\cong (M\toto M) \qquad G(\pi_M)\cong (M\times M\toto M)$$

Roughly speaking, every Lie groupoid emerges from a submersion $M\to S$ over a space $S$, which in this case is the manifold $N$, but in general it may be something singular, namely a differentiable stack. Later we will see how to formalize this idea (cf. \ref{presentations}).
\end{example}

% transitive groupoids
\begin{example}[Principal group-bundles]\label{exampbun}
Let $G$ be a Lie group, $N$ a manifold, and let $G\action P\to N$ be a smooth principal bundle.
This is essentially the same as a free proper action $G\action P$, for the surjective submersion $P\to N$ can be recovered as the quotient map $P\to P/G$ (see eg. \cite[App. E]{sharpe}).

The {\bf gauge groupoid} $P\times^{G } P \toto N$ consists of the equivariant isomorphisms between fibers of the principal bundle. It can be constructed as the quotient of the pair groupoid $P\times P\toto P$ by the action of $G$.
$$G \action P\to N \quad \leadsto \quad P\times^{G } P \toto N$$
Here we are considering the diagonal action $G\action P\times P$, which is also free and proper, 
we are writing $P\times^{G } P=(P\times P)/{G }$, and identifying $P/G\cong N$.
The structural maps of the pair groupoid are equivariant and that is why they induce a Lie groupoid structure in the quotient (cf. \ref{quotgrou}).

A gauge groupoid is {\bf transitive}, namely it has a single orbit, and its isotropy at any point is isomorphic to $G$. A map between principal bundles leads to a map between their gauge groupoids, but this assignation is not injective in general. %An element $h\in G $ induces a map $(G \action P)\to(G \action P)$ by $g\mapsto hgh^{-1}$, $p\mapsto hx$ which yields the identity between the gauge groupoids.
\end{example}

% linear groupoids
\begin{example}[Linear groupoids]\label{linear.groupoids}
The automorphisms $GL(V)$ of a vector space $V$ constitute the fundamental example of a Lie group. In a similar fashion, given $E\to M$ a smooth vector bundle, we can consider the {\bf general linear groupoid} $$E\to M \quad \leadsto \quad GL(E)\toto M$$
whose objects are the fibers of the vector bundle, and whose arrows are the linear isomorphisms between them. It can be defined as the gauge groupoid of the {\it frame bundle} of $E$.

This construction admits the usual variants. For instance, if the vector bundle is endowed with a metric we can define the {\bf orthogonal linear groupoid} $O(E)\toto M$, which consists of the isometries between the fibers.
\end{example}

\begin{example}[Orbifolds]\label{example.orbifolds}
Orbifolds are spaces locally modeled by quotients of euclidean spaces by finite group actions. We refer to \cite{mm} for a detailed treatment. During these notes we briefly discuss how orbifolds can be framed into the theory of Lie groupoids and differentiable stacks. 

Recall that an orbifold chart $(U,G ,\phi)$ on a space $O$
consists of a connected open $U\subset\R^d$ in some euclidean space, a finite group of $G\subset {\rm Diff}(U)$, and an open embedding $\phi:U/G \to O$.
An orbifold consists of a space $O$ endowed with an orbifold atlas, that is, a collection $\U=\{(U_i,G_i,\phi_i)\}_i$ of compatible orbifold charts. Two atlases define the same orbifold if they are compatible, namely if their union is again an atlas.

Given $O$ an orbifold and $\U$ a numerable atlas, we can define a Lie groupoid as follows.
$$O \ + \ \U \quad \leadsto \quad G\toto M$$
The manifold of objects is $M=\coprod_i U_i$. The manifold $G$ consists of germs of compositions of maps in some $G_i$, endowed with the sheaf-like manifold structure.
While the construction of this groupoid relies on the choice of an atlas, we will see that compatible atlases lead to {\it equivalent} Lie groupoids.
\end{example}

\subsection{Groupoid actions and representations}
\label{subsection.groupoid.actions}

% definitions and examples
Let $G\toto M$ be a Lie groupoid. Given $p:A\to M$ a smooth map, we can consider the good pullback
$G\times_M A=\{(g,a)|s(g)=p(a)\}$. A {\bf left groupoid action}
$$\theta:(G\toto M)\action (A\to M)$$
is a smooth map $\theta:G\times_M A\to A$, $(g,a)\mapsto \theta_g(a)$, such that $p(\theta_g(a))=t(g)$, $\theta_{1_x}=\id_{A_x}$ and $\theta_g\theta_h=\theta_{gh}$ when $g,h$ are composable. Right actions are defined analogously.
The map $p:A\to M$ is sometimes called the {\bf moment map} of the action.
An action $\theta$ realizes the arrows of the groupoid $G\toto M$ as symmetries of the family of fibers of the moment map, namely for each arrow $y\xfrom g x$ we have a diffeomorphism $\theta_g:A_x\to A_y$.

\begin{example}\

\begin{itemize}
\item
Actions of manifolds $(M\toto M)\action(A\to M)$ are trivial. Actions of Lie groups $(G\toto \ast)\action(A\to \ast)$ are the usual ones.
\item
An action $(G\ltimes M\toto M)\action(A\to M)$ of an action groupoid is the same as an action $G\action A$ and an equivariant map $A\to M$.
\item
For a submersion groupoid, an action $(M\times_N M\toto M)\action (A\to M)$ is the same as an equivalence relation $R$ on $A$ inducing a pullback square as bellow (cf. \ref{mapspbun}).
$$(M\times_N M\toto M)\action (A\to M) \quad \leftrightarrow \quad
\begin{matrix}
\xymatrix@R=10pt{A \ar[r] \ar[d] \ar@{}[dr]|(.7){\lefthalfcap}& A/R \ar[d]\\ M \ar[r] & N}
\end{matrix}$$
\end{itemize}
\end{example}

Given a groupoid action $G\action A$ we can construct the {\bf action groupoid} $G\ltimes A=(G\times_MA\toto A)$, on which the source is the projection, the target is the action, and composition, inverses and identities are induced by those of $G$. This generalizes the example \ref{grouacti}, and admits an obvious version for right actions.
We say that the action $G\action A$ is {\bf free} if the action groupoid has no isotropy, and that the action is {\bf proper} if the map $G \times_M A \to A\times A$, $(g,a)\mapsto (\theta_g(a),a)$ is so. 
By Godement criterion \ref{godement}, the orbit space of a free proper groupoid action inherits the structure of a manifold. This statement is in fact equivalent both to \ref{godement} and to \ref{anchinje}.

\bigskip

% action vs action maps
We can identify actions of $G$ with groupoid maps $\tilde G\to G$ of a special kind.
Given an action $(G\toto M)\action (A \to M)$, its moment map $A\to M$ and the projection $G\times_M A\to G$ define a groupoid map
$$p:(G\times_M A\toto A)\to (G\toto M)$$
inducing a pullback between the source maps. A map satisfying this property is called an {\bf action map}. Conversely, given any action map $p:(\tilde G\toto A)\to(G\toto M)$, the composition $G\times_M A\cong \tilde G\xto t A$ becomes a left action, we call it the {\bf underlying action}.
It is straightforward to check that these constructions are mutually inverse.
%We define the {\bf orbits} and the {\bf isotropy} of an action as those of the action groupoid. An action is {\bf free} if the action groupoid has no isotropy, and the action is {\bf proper} if the map $G\times_M A\to A\times A$, $(g,a)\mapsto (\theta_g(a),a)$, is so.
% \begin{example}
% Every Lie groupoid $G\toto M$ acts on its objects by $\theta_g(s(g))=t(g)$; the resulting action groupoid is isomorphic to $G\toto M$. Also, every Lie groupoid $G\toto M$ acts on its arrows with moment map the target, by the groupoid multiplication $\theta_g(h)=gh$.
% \end{example}
%More precisely, an {\bf action map} $p:(\tilde G\toto A)\to(G\toto M)$ is a groupoid map that induces a good pullback between the sources.
%$$\xymatrix{\tilde G \ar[r]^s \ar[d]_{p} \ar@{}[dr]|(.7){\lefthalfcap} & A \ar[d]^p \\ G\ar[r]_s & M}$$
%If $p$ is an action map, 
%The proof of next proposition is straightforward (cf. \cite{mkbook}).

\begin{proposition}\label{actiactm}
There is a 1-1 correspondence between left actions and action maps.
$$(G\toto M)\action (A\to M) \quad \leftrightarrow \quad (\tilde G\toto A)\to(G\toto M)$$
% $$\begin{matrix}
% G\times_M A\to A  & \mapsto  & (G\times_M A\toto A)\to (G\toto M) \\
% G\times_M A\cong \tilde G\xto t A & \leftmapsto & (\tilde G\toto A)\to (G\toto M) 
% \end{matrix}$$
\end{proposition}

A map of groupoids is an action map if and only if the associated target square is a pullback, thus we can also identify action maps with right actions.
% Given $G\toto M$ and $G'\toto M'$ Lie groupoids, an {\bf action}
% $(G\toto M)\action (G'\toto M')$ can be defined as a pair of actions $G\action G'$, $G\action M'$, such that the structural maps of $G'\toto M'$ are equivariant.

We study now a particular type of actions.
Given a Lie groupoid $G\toto M$ and a vector bundle $E\to M$, a {\bf linear representation} $(G\toto M)\action (E\to M)$ is an action $\theta:G\times_ME\to E$ such that for all $y\xfrom g x$ in $G$ the map $\theta_g:E_x\to E_y$ is linear.

\begin{example}\

\begin{itemize}
\item
Representations of manifolds $(M\toto M)\action (E\to M)$ are trivial. Representations of Lie groups $(G\toto \ast)\action(E\to\ast)$ are the usual ones.
\item
Representations of an action groupoid $(G\ltimes M\toto M)\action(E\to M)$ are equivariant vector bundles.
\item
Given $q:M\to N$ a surjective submersion, a representation $(M\times_N M\toto M)\action (E\to M)$ is the same as a vector bundle $\tilde E\to N$ such that $q^*\tilde E=E$ (cf. \ref{mapspbun}).
\end{itemize}
\end{example}

Using an exponential law argument it can be proved that 
representations are in 1-1 correspondence with maps on the general linear groupoid.
$$(G\toto M)\action (E\to M) \quad \leftrightarrow \quad (G\toto M)\to(GL(E)\toto M)$$
See \cite[Prop. 1.7.2]{mkbook} for the transitive case. The general case is proved analogously.

The action map of a representation can be regarded as a compatible diagram of Lie groupoids and vector bundles. More precisely, a {\bf VB-groupoid}
$$\xymatrix{\Gamma \ar@<.5ex>[r] \ar@<-.5ex>[r]  \ar[d] & E \ar[d]\\
G \ar@<.5ex>[r] \ar@<-.5ex>[r] & M}$$
is a groupoid map $(\Gamma\toto E)\to(G\toto M)$ such that $\Gamma\to G$ and $E\to M$ are vector bundles, and the structural maps of $\Gamma\toto E$ are vector bundle maps.
The {\bf core} $C\to M$ is defined by $C=\ker(s:\Gamma\to E)|_M$, where we are identifying $M=u(M)$.

\begin{example}
Given $G\toto M$ we can construct a new Lie groupoid $TG\toto TM$ whose structural maps are the differentials of those of $G$. The canonical projections define the {\bf tangent VB-groupoid}.
$$(TG\toto TM)\to(G\toto M)$$
The core of this VB-groupoid is the {\bf Lie algebroid} $A_G$ associated to $G$ (cf. \cite{mkbook}, \cite{mm}).
\end{example}

When the core is trivial, namely $C=0_M\to M$, then the ranks of $\Gamma\to G$ and $E\to M$ agree and we can identify the total space $\Gamma$ with the pullback of $E$ along the source map, $\Gamma\cong G\times_M E$. Thus \ref{actiactm} provides the following.

\begin{proposition}\label{vbgroupo}
There is a 1-1 correspondence between representations of $G\toto M$ and VB-groupoids 
$(\Gamma\toto E)\to(G\toto M)$ with trivial core.
\end{proposition}

VB-groupoids are something more general than representations. Actually, they admit a nice 
interpretation in the theory of {\it representations up to homotopy} (cf. \cite{gsm}).
For more on VB-groupoids and their infinitesimal counterpart we refer to \cite{bcd}.

%%%%%%%%%%%%%%%%%%%%%%%%%%%%%%%%%

\subsection{The normal representation}\label{section.normal.representation}

Given $G\toto M$ a Lie groupoid and $O\subset M$ an orbit, the {\it normal representation} is a representation of the restriction $G_O\toto O$ over the normal bundle $NO\to O$. It encodes the linear infinitesimal information around the orbit and plays a fundamental role in the theory. We present it here after a short digression on restrictions.

\bigskip

Given $G\toto M$ a Lie groupoid and $A\subset M$ a submanifold, the subset $G_A\subset G$ may not be a submanifold in general, and even if that is the case, $G_A\toto A$ may not be a Lie groupoid.

\begin{example}
Let $G\toto M$ be the Lie groupoid arising from the projection $S^1\times\R\to S^1$ (cf. \ref{examsubm}). Let $A=\{(e^{it^2},t):t\in\R\}$ and $B=\{(e^{it^3},t):t\in\R\}$. Then $A,B\subset M$ are embedded submanifolds, but $G_A\subset G$ is not a submanifold, and even when $G_B\subset G$ is embedded, the restriction of the source map $G_B\to B$ is not a submersion and hence $G_B\toto B$ is not a Lie groupoid.
\end{example}

We say that the restriction $G_A\toto A$ is {\bf well-defined} when $G_A\subset G$ is a submanifold, $G_A\toto A$ is a Lie groupoid, and the following is a good pullback of manifolds.
$$\xymatrix{G_A \ar[r] \ar[d]  \ar@{}[dr]|(.7){\lefthalfcap}& G \ar[d] \\ A\times A\ar[r] & M\times M}$$
For instance, given $U\subset M$ open, the restriction $G_U\toto U$ is clearly well-defined.

\begin{proposition}
Given $G\toto M$ a Lie groupoid and $O\subset M$ an orbit, the restriction $G_O\toto O$ is well-defined.
\end{proposition}

\begin{proof}
The key point here is that an orbit $O\subset M$ is an {\it initial submanifold}, namely every smooth map $Z\to M$ whose image lies in $O$ restricts to a smooth map $Z\to O$. This is because $t_x:G(-,x)\to M$ has constant rank, hence a map $Z\to M$ with image included in $O$ can be locally lifted to a map $Z\to G(-,x)$, proving that the co-restriction $Z\to O$ is also smooth.

Now, since $G_O=s^{-1}(O)=t^{-1}(O)$, we can lift the manifold structure in $O\subset M$ to one in $G_O\subset G$ that makes it an initial submanifold of the same codimension. The left square below is a good pullback by construction. It easily follows from this that the right square is a good pullback as well.
$$\xymatrix{G_O \ar[r]^s \ar[d]  \ar@{}[dr]|(.7){\lefthalfcap}& G \ar[d] \\ O\ar[r]_s & M}
\qquad
\xymatrix{G_O \ar[r] \ar[d]  \ar@{}[dr]|(.7){\lefthalfcap}& G \ar[d] \\ O\times O\ar[r] & M\times M}$$
\end{proof}

Given $G\toto M$ and $O\subset M$, we can construct a sequence of VB-groupoids
$$(TG_O\toto TO) \quad \to \quad (TG|_{G_O}\toto TM|_O) \quad \to \quad (N(G_O)\toto NO)$$
where the first one is the tangent of $G_O$, the second one is the {\it restriction} of the tangent of $G$ to $G_O$, and the third one, the {\bf normal bundle}, is defined by the quotient vector bundles, with the induced structural maps.

The submanifolds $G_O\subset G$ and $O\subset M$ have the same codimension, 
then the ranks of the vector bundles $N(G_O)$ and $NO$ agree, 
the core of $N(G_O)\toto NO$ is trivial
and there is an underlying groupoid representation (cf. \ref{vbgroupo}).
$$\eta:(G_O\toto O)\action(NO\to O)$$
This is called the {\bf normal representation} of $G$ at the orbit $O$.

Unraveling this construction, the normal representation can be geometrically described as follows:
if $\gamma$ is a curve on $M$ whose velocity at 0 represents $v\in N_xO$, and $\tilde\gamma$ is a curve on $G$ such that $\tilde\gamma(0)=g$ and $s\circ\tilde\gamma=\gamma$, then $\eta_g(v)\in N_yO$ is defined by the velocity at 0 of $t\circ\tilde\gamma$.

\bigskip

Fixed $x\in M$, the normal representation can be restricted to the isotropy group, say $\eta_x:G_x\action N_xO$. 
For some purposes, the restriction $\eta_x$ manages to encode the necessary information of $\eta$.
Note that if $x,y$ belong to the same orbit, then an arrow $y\xfrom g x$ yields an isomorphism of group representations $(G_x\action N_xO) \cong (G_y\action N_yO)$.

% \begin{example}\
% 
% \begin{itemize}
%  \item 
% When there is no isotropy then the normal representation reduces to the normal bundle of the orbit. For a unit groupoid $M\toto M$ we have $\eta_x = \ast \action T_xM$. More generally, for a submersion groupoid $M\times NM\toto M$ we have $\eta_x =\ast\action N_xF\cong T_{q(x)}N$.
%  \item
% For Lie groups, and more generally transitive groupoids, there is only one orbit and its normal bundle is zero. Thus, the normal representation is in this cases is the trivial action of the groupoid $G_O\toto O$.
%  \item
% A more interesting example is that of action groupoids. In this case, the normal representation is the linear model for the action of the isotropy on a transverse to the orbit.
% \end{itemize}
% \end{example}

\begin{remark}
The normal representation is functorial. If 
$\phi:(G\toto M)\to(G'\toto M')$ is a map sending an orbit $O\subset M$ to $O'\subset M'$, then we have a naturally induced morphism of VB-groupoids and of representations
$\phi_*:(G_O\action NO)\to(G'_{O'}\action NO')$.
In particular, for each $x\in O$, there is a morphism of Lie group representations $\phi_*:(G_x\action N_xO)\to (G_{\phi(x)}\action{N_{\phi(x)}}O')$.
\end{remark}

%%%%%%%%%%%%%%%%%%%%%%%%%%%%%%%%%%

\subsection{The anchor map}

Given $G\toto M$ a Lie groupoid, its {\bf anchor} $\rho_G=(t,s):G\to M\times M$ is the map whose components are the source and the target, namely $\rho(y\xfrom g x)=(y,x)$.
The image of $\rho_G$ is the equivalence relation on $M$ defining the orbit space $M/G$, and its fiber over a diagonal point $(x,x)$ is the isotropy group $G_x$.

%We can use the normal representation to characterize the differential of the anchor at a point.

\begin{proposition}\label{diffanch}
Given $y\xfrom g x$ in $G$, the differential of the anchor $d_g\rho:T_gG\to T_yM\times T_xM$ has kernel and image given by
$$\ker( d_g\rho) = T_gG(y,x) \qquad \im( d_g\rho) = \{(v,w)| [v]=\eta_g[w] \}$$
\end{proposition}

\begin{proof}
The description of the kernel follows from the fact that $t:G(-,x)\to M$ has constant rank and fiber $G(y,x)$.
Regarding the image, note that the map $d_g\rho$ yields a commutative square
$$\xymatrix{T_gG \ar[r] \ar[d]& T_yM\times T_xM \ar[d]\\ N_gG_O \ar[r] & N_yO\times N_xO}$$
from which any vector in ${\rm Im}(d_g\rho)$ satisfies the equation involving the normal representation.
The other inclusion follows by an argument on the dimensions:
the fibration $G(y,x)\to G(-,x) \to O_x$ implies that
$$\dim \ker( d_g\rho) =\dim G(y,x)=\dim G- \dim M -\dim O$$
Then we conclude
$$\codim \im( d_g\rho)= 2\dim M - \dim G + \dim \ker( d_g\rho)= \dim M - \dim O$$
\end{proof}

Previous proposition plays a role in many results.
An immediate corollary is that the anchor is injective if and only if it is an injective immersion, and it is surjective if and only if it is a surjective submersion. % Thus, if the anchor is surjective it has to be open as well.
Next we provide characterizations both for submersion groupoids and for gauge groupoids.

\begin{proposition}\label{anchinje}
The submersion groupoid construction (cf. \ref{examsubm}) provides a 1-1 correspondence between surjective submersions and Lie groupoids $G\toto M$ with anchor closed and injective.
$$\begin{matrix}
G\toto M \qquad  & \mapsto &\qquad M\to M/G \\
 M\times_N M\toto M \qquad & \leftmapsto &\qquad M\to N
\end{matrix}$$
\end{proposition}
\begin{proof}
Given a surjective submersion $q:M\to N$, its submersion groupoid $M\times_N M\toto M$ has trivial isotropy and Hausdorff orbit space, hence its anchor is injective and closed.
Conversely, given $G\toto M$ whose anchor is closed and injective, it follows from \ref{diffanch} that the anchor is also an immersion, hence a closed embedding, and we can use Godement criterion \ref{godement} to endow the quotient $M/G$ with a manifold structure.
These constructions are mutually inverse up to obvious isomorphisms.
\end{proof}

%\begin{example}
%The Kronecker foliation on the Torus (cf. \ref{kronecker}) leads to an example on which the anchor is injective but not closed neither an embedding. If the anchor is not closed but an embedding, then $M/G$ fails to be Hausdorff, see for instance the groupoid emerging from the relation $t\sim -t$, $|t|\neq 1$, on $\R\setminus\{0\}$. 
%\end{example}

%We say that $G$ is a {\bf transitive groupoid} if $\rho_G$ is surjective, or equivalently if $G$ has only one orbit, or if $M/G\cong\ast$.

\begin{proposition}\label{anchsurj}
The gauge construction presented in \ref{exampbun} provides a 1-1 correspondence between transitive Lie groupoids $G\toto M$ and principal group-bundles $G\action P\to M$.
$$\begin{matrix}
G\toto M \qquad  & \mapsto & \qquad G_x\action G(-,x)\to M \\
 P\times^H P \toto M \qquad & \leftmapsto & \qquad H\action P\to M
\end{matrix}$$
\end{proposition}

\begin{proof}
To every principal bundle $G\action P\to M$  we can associate its gauge groupoid $P\times^GP\toto P$, which is clearly transitive. Conversely, given $G\toto M$ a transitive Lie groupoid, by fixing some $x\in M$ we can associate to it the principal bundle $G_x\action G(-,x)\xto{t_x} M$. Note that since the anchor is a submersion by \ref{diffanch}, the map $t_x:G(-,x)\to M$ also is a submersion and both manifold structures on $M$, the original and that of the orbit, agree.
It is easy to check that these constructions are mutually inverse up to isomorphism.
% Starting with a principal bundle, the choice of a point $p\in P$ over $x\in M$ yields an isomorphism
% $$(G\action P\to M) \quad \to\quad (G\action P\times^GP(x,-)\toto M) \qquad p' \mapsto [p',p]$$
\end{proof}

\begin{remark}
In \ref{anchinje} the 1-1 correspondence extends to maps. This may be understood as a reformulation of Godement criterion.
On the other hand, the correspondence in \ref{anchsurj} does not preserves maps. In order to get a principal bundle out of a transitive groupoid we need to pick an arbitrary object, and general maps need not to respect this choice.
\end{remark}

%%%%%%%%%%%%%%%%%%%%%%%%%%%%%%%%%%

\subsection{Principal groupoid-bundles}\label{section principal bundles}

Groupoid-bundles are a natural generalization of group-bundles, on which much of the theory can be reconstructed. In this subsection we give the definition and the basic properties.

\bigskip
\def\xaction#1{\overset{#1}{\action}}

% groupoid bundles
Let $G\toto M$ be a Lie groupoid and let $N$ be a manifold.
A {\bf left $G$-bundle} $G\action P\to N$ consists of a left action 
$\theta:G\action P$ and a surjective submersion
$q:P\to N$ such that the fibers of $q$ are invariant by $\theta$, namely $q(\theta_g(x))=q(x)$ for all $(g,x)\in G\times_M P$.
There is a canonical map from the action groupoid to the submersion groupoid,
$$\xi:(G\times_M P\toto P) \to (P\times_N P\toto P) \qquad
(\theta_g(x)\xfrom{(g,x)}x)\mapsto(\theta_g(x),x)$$
A bundle $G\action P\to N$ is called {\bf principal} if the action is free and the orbits are exactly the fibers of the submersion. Note that, in view of \ref{anchinje}, the bundle is principal if and only if $\xi$ is an isomorphism.

Right bundles $N \from P \raction G$, as well as the corresponding notions, are defined analogously.
 
% to be bijective, the anchor $\rho_{G\rtimes P}$ to be a closed inclusion, and the groupoid (cf. \ref{equimani}).
% if the previous map is an isomorphism, or equivalently, if the action is free and its orbits are exactly the fibers of the submersion, for these two conditions force the map $(g,x)\mapsto(\phi_g(x),x)$ to be bijective, the anchor $\rho_{G\rtimes P}$ to be a closed inclusion, and the groupoid (cf. \ref{equimani}).

\begin{example}\

\begin{itemize}
\item Principal $(G\toto\ast)$-bundles are the usual principal group-bundles.
\item A $(M\toto M)$-bundle is the same as a pair of maps $M\from P\to N$ where the second leg is a surjective submersion. It is principal if and only if $P\to N$ is a diffeomorphism.
\item Given a group action $G\action M$, a principal $G\ltimes M$-bundle is the same as a principal $G$-bundle $G\action P\to N$ and an equivariant map $P\to M$.
\end{itemize}
\end{example}

In a principal bundle $G\action P\to N$ the action $\theta$ is free and proper.
Conversely, given a free proper action $G\action P$, it can be seen that the action groupoid $G\ltimes P$ has a smooth quotient $P/G$ (cf. \ref{anchinje}) and therefore $G\action P\to P/G$ is a principal bundle. 
In other words, the submersion $q$ is implicit in the action $\theta$, as it happens in the group case (cf. \cite[App. E]{sharpe}). Thus we have

\begin{proposition}\label{charpbun}
There is a 1-1 correspondence between principal $G$-bundles and free proper actions of $G$.
\end{proposition}

%Last proposition can be seen as a reformulation of Godement criterion.

% Last proposition is in fact a reformulation of Godement criterion. Assuming the proposition is true, and given a relation $R$ on $M$ satisfying the hypothesis of Godement criterion, there is a Lie groupoid structure on $R\toto M$, whose canonical action on objects is free and proper, hence the quotient $M/R$ is smooth and the projection $M\to M/R$ is a submersion.

% Last proposition can be thought of as a reformulation of \ref{godement}. In fact, assuming the proposition, and given a relation $R$ on $M$ satisfying Godement criterion, then pair groupoid $M\times M\toto M$ restricts to a Lie groupoid structure on $R\toto M$, whose canonical action on objects is free and proper, hence the quotient $M\to M/R$ is smooth.

A {\bf map of bundles} $\phi:(G\action P' \to N')\to (G\action P\to N)$ is a smooth map $\phi:P'\to P$ compatible with the actions and the submersions.

Given $G\action P\to N$ a principal bundle and $N'\to N$ a smooth map, we define the {\bf pullback bundle} by
$$G\action (P\times_N N')\to N' \qquad \theta'_g(x,y)=(\theta_g(x),y) \quad q'(x,y)=y$$
which is also principal. It is easy to see that the canonical projection $P\times_N N'\to P$ is a map of bundles. Conversely, every map of principal bundles turns out to be a pullback.

\begin{proposition}\label{mapspbun}
A map $\phi:(G\action P' \to N')\to (G\action P\to N)$ of principal bundles induces an isomorphism of bundles
$$(G\action P'\to N')\cong (G\action (P\times_{N}N')\to N')$$
\end{proposition}

\begin{proof}
Again we can imitate the Lie group case. 
It is enough to show that $\phi$ gives diffeomorphisms between the fibers $P'_{a'}\xto\cong P_a$, where $a'\in N'$ and $\phi(a')=a\in N$.
Choosing $u'\in P'_{a'}$ and calling $u=\phi(u')\in P_a$, we can identify
$$G(-,x)\cong P'_{a'} \quad g\mapsto \theta'_g(u') \quad \text{ and } \quad G(-,x)\cong P_a\quad  g\mapsto \theta_g(u)$$
and under these identifications, the map is just the identity, from where the result is clear.
\end{proof}

% For each $a'\in P'$, writing $a=\phi(a')$, the fibers $P'_{q'(a')}$ and $(P\times_{N}N')_{q'(a')}=P_{q(a))}$ fit into a square
% $$\begin{matrix}
% G(x,-)\times \{a'\} & \xto{\theta'} & P'_{q'(a')}\\
% \downarrow & & \downarrow \\
% G(x,-)\times \{a\} & \xto{\theta} & P_{q(a)}
%   \end{matrix}$$
% from where the result is clear.
% \end{proof}

Given $G\toto M$ a Lie groupoid, the {\bf unit principal bundle} $G\action G\to M$ is defined by the  action $\theta_g(g')=gg'$ and the quotient map $q(g)=s(g)$. We say that a principal bundle $G\action P\to N$ is {\bf trivial} if it is the pullback of the unit bundle along some map $N\to M$. Note that a trivial principal bundle need not to be a trivial map.

Any principal bundle $G\action P\to N$ is locally trivial. In fact, if $U\subset N$ is a small open and $\sigma:U\to P$ is a local section of $q$, then there is an isomorphism between the restriction to $U$ and the trivial bundle induce by $t\sigma$.
$$(G\action G\times_M U \to U)\cong(G\action q^{-1}(U)\to U) \qquad (g,u)\mapsto g\sigma(u)$$
This leads to a cocycle description of principal bundles, completely analog to the group case (see \ref{cocycle}).
In particular, a principal bundle is trivial if and only if it admits a global section.

\begin{remark}\label{gauge}
Given $G\action P\to N$ a principal bundle such that $P\to M$ is a submersion, we can define the diagonal action $(G\toto M)\action(P\times_M P\to M)$, which is also free and proper. The structural maps of the submersion groupoid $P\times_M P\toto P$ are equivariant, thus they induce maps in the quotients $P\times_M^G P=(P\times_M P)/G$ and $P/G=N$, defining a new Lie groupoid, the {\bf gauge groupoid}.
$$G\action P\to N \quad \leadsto \quad %((P\times_M P)/G\toto P/G)=
(P\times_M^G P \toto N)$$
This construction generalizes both the one presented in \ref{examsubm} and that in \ref{exampbun}.
\end{remark}

%%%%%%%%%%%%%%%%%%%%%%%%%%%%%%%%%%
%%%%%%%%%%%%%%%%%%%%%%%%%%%%%%%%%%

\section{Equivalences}\label{section-equivalences}

We start this section by discussing isomorphisms between Lie groupoid maps.
Then we deal with Morita maps and provide an original characterization for them.
After that we make a short digression on homotopy pullbacks, which play an important role hereafter. 
We define equivalent groupoids and generalized maps by using Morita maps, and explain the relation of this approach to that of principal bibundles. 
Finally, we introduce differentiable stacks, showing that a Lie groupoid is essentially the same as a presentation for such a stack.

Some references for this material are \cite{mm} and \cite{mrcun}.

\subsection{Isomorphisms of maps}

% isomorphisms of maps
The category of Lie groupoids and maps can be enriched over groupoids, namely there are isomorphisms between maps, and many times it is worth identifying isomorphic maps, and considering diagrams that  commute up to isomorphisms.

\bigskip

Given $\phi_1,\phi_2:(G\toto M)\to (G'\toto M')$ maps of Lie groupoids, an {\bf isomorphism} $\alpha:\phi_1\cong\phi_2$ consists of a smooth map $\alpha:M\to G'$
assigning to each object $x$ in $M$ an arrow $\phi_2(x)\xfrom{\alpha(x)}\phi_1(x)$ in $G'$
such that $\alpha(y)\phi_1(g)=\phi_2(g)\alpha(x)$ for all $y \xfrom g x$.
%$$\xymatrix{\phi_1(x) \ar[d]_{\phi_1(g)} \ar[r]^{\alpha(x)} & \phi_2(x) \ar[d]^{\phi_2(g)} \\ \phi_1(y) \ar[r]^{\alpha(y)} & \phi_2(y)}$$

\begin{example}\
 
\begin{itemize}
 \item Two maps between manifolds are isomorphic if and only if they are equal.
 \item Two maps between Lie groups $\phi_1,\phi_2:(G\toto \ast)\to (G'\toto \ast)$ are isomorphic if and only if they differ by an inner automorphisms of $G'$.
 \item Two maps between submersion groupoids  are isomorphic if and only if they induce the same map in the orbit manifolds (cf. \ref{anchinje}).
$$\phi_1\cong\phi_2:(G\toto M)\to (G'\toto M') \quad \iff \quad (\phi_1)_*=(\phi_2)_*:M/G\to M'/G'$$
\end{itemize}
\end{example}

For the sake of simplicity we will identify two maps if they are isomorphic, and do not pay attention to the automorphisms of a given map.
%forget about the automorphisms of a given map, and just .
We denote the set of isomorphism classes of maps $G\to G'$ by ${\rm Maps}(G,G')/_\cong$ 
and the category of Lie groupoids and isomorphism classes of maps by
$$\{\text{Lie Groupoids}\}/_{\simeq}$$
A map which is invertible up to isomorphism is called a {\bf categorical equivalence}, and an inverse up to isomorphism is called a {\bf quasi-inverse}. We use the notation $\simeq$ for categorical equivalences and keep $\cong$ for the isomorphisms.

\bigskip

% Groupoid of arrows
Given $G\toto M$, its {\bf groupoid of arrows} $G^I=(G\times_M G\times_M G\toto G)$
is the Lie groupoid whose objects are the arrows of $G$ and whose arrows are commutative squares, or equivalently chains of three composable arrows.
%$(g'h^{-1},h,h^{-1}g)$
%$$(y' \xfrom{g'} x') \xfrom{(g'h^{-1},h,h^{-1}g)}(y\xfrom g x)$$
$$\begin{matrix}
\xymatrix{
y \ar[d]_{g'h^{-1}}& x \ar[d]^{h^{-1}g} \ar[l]^{g}  \\
y' & x' \ar[l]^{g'} \ar[lu]^h}
\end{matrix} \qquad \longleftrightarrow \qquad
(y' \xfrom{g'} x') \xfrom{(g'h^{-1},h,h^{-1}g)}(y\xfrom g x)$$
With these definitions we can regard the unit, source and target of $G$ as Lie groupoid maps
$u:G\to G^I$ and $s,t:G^I\to G$.
There is a tautological isomorphism $s\cong t$ given by the identity $G\to G$, which is {\it universal}: An isomorphism $\alpha:\phi_1\cong\phi_2:G'\to G$ turns out to be the same as a map $\tilde\alpha:G' \to G^I$ such that $\phi_1=s\tilde\alpha$ and $\phi_2=t\tilde\alpha$.

\bigskip

% behavior
We have associated to a Lie groupoid $G\toto M$ its orbit space $M/G$ and its normal representations $G_x\action N_xO$, $x\in M$. These constructions are functorial, and behave well with respect to isomorphisms of maps.

\begin{proposition}\label{behavior.isomorphic.maps}
If $\alpha:\phi_1\cong\phi_2:(G\toto M)\to(G'\toto M')$ then $\phi_1,\phi_2$ induce the same map
$(\phi_1)_*=(\phi_2)_*:M/G\to M'/G'$ between the orbit spaces, and there are commutative triangles between the normal representations
% 
% .
% And regarding the normal representation, if $\alpha:\phi_1\cong\phi_2$ is an isomorphism, then their the isotropy groups and the normal representations fit in the following commutative triangle.
% $$\xymatrix{ & G_x \ar[dl]_{(\phi_1)_*} \ar[dr]^{(\phi_2)_*}& \\ G'_{\phi_1(x)} \ar[rr]_{g\mapsto \alpha(x)^{-1}g\alpha(x)} & & G'_{\phi_2(x)}}$$
$$\xymatrix@C=0pt{
& G_x\action N_xO \ar[dl]_{(\phi_1)_*} \ar[dr]^{(\phi_2)_*} & \\
G'_{\phi_1(x)}\action N_{\phi_1(x)}O' \ar[rr]_{\eta'_{\alpha(x)}} & & G'_{\phi_2(x)}\action N_{\phi_2(x)}O'}$$
\end{proposition}

The proof is straightforward. Note that it is enough to prove it for the universal isomorphism $s\cong t:G^I\to G$.

%It follows that a categorical equivalence induces a homeomorphisms between the orbit spaces and isomorphisms between the normal representations.

%In light of previous proposition, this follows in general once it is proved for the canonical isomorphism $s\cong t:G^I\to G$, which is a simple exercise.

%%%%%%%%%%%%%%%%%%%%%%%%%%%%%%%%%%

\subsection{Morita maps}

%intro
A Lie groupoid $G\toto M$ can be regarded as a presentation for a differentiable stack, which we will denote by $[M/G]$.
In order to formalize this we use {\it Morita maps}.
Intuitively, a Morita map is a map of groupoids inducing isomorphisms between the underlying stacks.

\bigskip

% definitions
Let $\phi:(G\toto M)\to (G'\toto M')$ be a map between Lie groupoids. Then $\phi$ is {\bf fully faithful} if it induces a good pullback of manifolds between the anchors,
$$\xymatrix{
G \ar[r]^\phi \ar[d]_{\rho} \ar@{}[dr]|(.7){\lefthalfcap} & G' \ar[d]^{\rho'}\\
M\times M \ar[r]^{\phi\times\phi} & M'\times M'}$$
and $\phi$ is {\bf essentially surjective} if the following map of manifolds is a surjective submersion.
$$ t \pr_1: G'\times_{M'} M \to M' \qquad (x'\xfrom{g'}\phi(x),x) \mapsto x'$$
We say that $\phi$ is a {\bf Morita map} if it is both fully faithful and essentially surjective. We use the notation $\sim$ for Morita maps.

\begin{example}\label{example.weak.equivalence}\

 \begin{itemize}
  \item A map between manifolds is fully faithful if and only if it is an injective immersion, and it is essentially surjective if and only if it is a surjective submersion. Thus, in this case, a Morita map is the same as a diffeomorphism.
%  \item Given $G$ a Lie groupoid, it easily follows from the definitions that the source and target, regarded as groupoid maps $s,t:G^I\to G$, are weak equivalences.
  \item If $G$ is transitive, then any map $G'\to G$ is essentially surjective.
  \item Given $G\toto M$ and $A\subset M$ such that the restriction $G_A\toto A$ is well-defined (cf. \ref{section.normal.representation}), the inclusion $(G_A\toto A)\to (G\toto M)$ is fully faithful, and it is essentially surjective if and only if $A$ intersects transversally every orbit.
  \item If $O$ is an orbifold and $\U,\U'$ are numerable atlases of $O$ such that $\U$ refines $\U'$, then a choice of inclusions leads to a Lie groupoid map $(G\toto M)\to (G'\toto M')$ between the induced Lie groupoids (cf. \ref{example.orbifolds}). This map is Morita (cf. \ref{charequi}, see also \cite{mm}).
%   Since any two atlases admit a common refinement we can say that associated to each orbifold we have a Lie groupoid well-defined up to equivalence.
 \end{itemize}
\end{example}

% categorical vs weak equivalences
Every isomorphism is a categorical equivalences, and every categorical equivalence is a Morita map. Actually, if two maps are isomorphic, then one of them is Morita if and only if the other is so. This can be proved directly from the definitions or can be obtained as a corollary of \ref{behavior.isomorphic.maps} and \ref{charequi}.
Next example shows that the three notions are in fact different.

\begin{example}
Let $q:M\to N$ be a submersion and let $\phi$ be the induced map
$$\phi:(M\times_N M \toto M)\to(N\toto N)$$
Then $\phi$ is Morita if and only if $q$ is surjective, $\phi$ is a categorical equivalence if and only if $q$ admits a global section, and $\phi$ is an isomorphism if and only if $q$ is a diffeomorphism. 
\end{example}
% 
% \begin{remark}
% Our requirement of a good pullback when defining fully faithful maps leads to a slightly different definition than the one in the literature (see e.g. \cite{mm}). For instance, the map $\R\to\R$, $x\mapsto x^3$, is not fully faithful for us, while it is for the other definition.
% Nevertheless, if $\phi$ is essentially surjective then the maps $\phi\times\phi$ and $\rho'$ are transversal, therefore the pullback formed by the anchors has to be good (cf. \ref{pullback of manifolds}) and our notion of weak equivalence agrees with the usual one.
% \end{remark}

% If $\phi:G\to G'$ is fully faithful, then the pullback involving the anchors induces diffeomorphisms between their fibers, namely $G(x,y)\xto\cong G'(x',y')$. This implies firstly that we have isomorphisms on the isotropy groups $G_x\xto\cong G'_{x'}$, and secondly that the induced map $\phi_*:M/G\to M'/G'$ is injective.

\begin{remark}\label{categorical.monomorphism}
When working up to isomorphisms, a fully faithful map $\phi:G\to G'$ is a {\it categorical monomorphism}, namely for any $H$ it induces an injective map
$${\rm Maps}(H,G)/_\cong \to {\rm Maps}(H,G')/_\cong \quad \psi\mapsto \phi\psi$$
In fact, given $\psi_1,\psi_2:H\to G$, an isomorphism $\alpha:\phi\psi_1\cong\phi\psi_2$ can be lifted to another $\tilde\alpha:\psi_1\cong\psi_2$ by the universal property of the pullback determined by the anchors.
\end{remark}
% It follows that the {\bf kernel} $(K\toto M)=\phi^{-1}(M'\toto M')$ of a surjective equivalence is a Lie groupoid. By Godement criterion we have the following.
% 

\begin{remark}
Our definition of fully faithful maps slightly differs from the one in the literature  (cf. eg. \cite{mm}), for we are asking for the pullback to be good. As an example, the map $\R\to\R$, $x\mapsto x^3$, is not fully faithful for us, while it is for the other definition. Nevertheless, when the map is essentially surjective, then $\rho'$ and $\phi\times\phi$ are transverse and both definitions for Morita maps agree.%lead to the same notion of weak equivalence.
%Nevertheless, if $\phi$ is essentially surjective then the maps $\phi\times\phi$ and $\rho'$ are transversal, therefore the pullback formed by the anchors has to be good (cf. \ref{pullback of manifolds}) and our notion of weak equivalence agrees with the usual one.
\end{remark}

%%%%%%%%%%%%%%%%%%%%%%%%%%%%%%%

\subsection{A characterization of Morita maps}

% characterization
The differentiable stack $[M/G]$ associated to a Lie groupoid $G\toto M$ consists of the orbit space $M/G$ endowed with some smooth data, encoded in the normal representations $G_x\action N_xO$. These representations play the role of tangent spaces of $[M/G]$. 
Next criterion can be seen as a formulation of this idea.

\begin{theorem}\label{charequi}
A map $\phi:(G\toto M)\to (G'\toto M')$ is Morita if and only if it yields a homeomorphism between the orbit spaces and isomorphisms between the normal representations.
$$\phi:(G\toto M)\xto\sim(G'\toto M') \iff 
\begin{cases}
\phi_*:M/G\xto\cong M'/G' \\
\phi_x:(G_x\action N_xO)\xto\cong (G'_{x'}\action N_{x'}O') \text{ for all } x
\end{cases}$$
% \begin{itemize}\itemsep=0pt
% \item $\phi_*:M/G\xto\cong M'/G'$ is a homeomorphism, and
% \item $\phi_*:(G_x\action NO_x)\xto\cong (G'_{\phi(x)}\action NO'_{\phi(x)})$ is an isomorphism for all $x$.
% \end{itemize}
\end{theorem}

\begin{proof}
Let us write $x'=\phi(x)$, $g'=\phi(g)$, and so on.

\smallskip

\noindent{\it First step:} Assume that $\phi$ is fully faithful. The pullback involving the anchors induces diffeomorphisms between their fibers, namely 
$$G(y,x)\xto\cong G'(y',x') \qquad \text{ for all } x,y\in M$$
This implies that we have isomorphisms on the isotropy groups $G_x\xto\cong G'_{x'}$, and that the induced map $\phi_*:M/G\to M'/G'$ between the orbit spaces is injective.
Moreover, the map $\phi_*:N_xO\to N_{x'}O'$ is a monomorphism for all $x$.
In fact, if $v\in T_xM$ and $d_x\phi(v)\in T_{x'}O'$, then $(d_x\phi(v),0)\in{\rm Im}( d_{u(x)}\rho')$ (cf. \ref{diffanch}), and since the following is a pullback of vector spaces,
$$\xymatrix{
T_{u(x)}G \ar[r] \ar[d]_{d_{u(x)}\rho} \ar@{}[dr]|(.7){\lefthalfcap} & T_{u(x')}G' \ar[d]^{d_{u(x')}\rho'}\\
T_xM\times T_xM \ar[r] & T_{x'}M'\times T_{x'}M'}$$
it turns out that $(v,0)\in{\rm Im}( d_{u(x)}\rho)$, hence $v\in T_xO$.

\smallskip

\noindent{\it Second step:}
It is easy to see that $t\pr_1:G'\times_{M'}M\to M'$ is surjective if and only if $\phi_*:M/G\to M/G'$ is so. 
% It also turns out that $t\pr_1$ is a submersion if and only if $\phi_*:N_xO\to N_{x'}O'$ is an epimorphism for all $x$. This follows by diagram chasing down below (cf. \ref{diffanch}).
Now consider the following diagram of vector spaces.
$$\xymatrix{ T_{g'}(G'\times_{M'}M) \ar[d] \ar[r] \ar@{}[dr]|(.7){\lefthalfcap} & T_{g'}G' \ar[r]^{dt} \ar[d]_{ds} & T_{y'}M' \ar[d]^{\eta_{g'^{-1}}\pi_{y'}} \\
T_xM \ar[r] & T_{x'}M' \ar[r]^{\pi_{x'}} & N_{x'}O' }$$
It follows by diagram chasing and \ref{diffanch} that
the upper composition is an epimorphism if and only if the lower one is so. We can conclude that $t\pr_1$ is a submersion if and only if $\phi_*:N_xO\to N_{x'}O'$ is an epimorphism for all $x$.

Note that if $\phi$ is essentially surjective then $\phi_*:M/G\to M'/G'$ has to be open, for the top and the right arrow in next commutative square are so.
$$\xymatrix{G'\times_{M'}M \ar[r]^{t\pr_1} \ar[d]_{q\pr_2} & M' \ar[d]_{q'} \\ M/G\ar[r]_{\phi_*} & M'/G'} $$

\smallskip

\noindent{\it Third step:} It only remains to show that the criterion implies that $\phi$ is fully faithful. From the isomorphisms $G_x\xto\cong G'_{x'}$ and the homeomorphism $M/G\xto\cong M'/G'$ it follows that the anchor maps define a set-theoretical pullback, and that the maps $G(y,x)\xto\cong G'(y',x')$ are diffeomorphisms.
Since the maps $N_xO\to N_{x'}O'$ are onto, the following are transverse (cf. \ref{diffanch}).
$$\phi\times\phi:M\times M \to M'\times M' \qquad \rho':G'\to M'\times M'$$
Thus their pullback exists and it is good, and we get a map $G\to (M\times M)\times_{M'\times M'}G'$. It is bijective for both are set-theoretical pullbacks. Its differential at each point is invertible as can be proved by diagram chasing between the following two exact sequences (cf. \ref{diffanch}).
$$\xymatrix{
0 \ar[r] & T_gG(y,x) \ar[d]^\cong \ar[r] & T_gG \ar[r] \ar[d] & T_yM\times T_xM \ar[r] \ar[d] & N_yO \ar[r] \ar[d]^\cong & 0  \\ 
0 \ar[r] & T_{g'}G'(y',x') \ar[r] & T_{g'}G' \ar[r] & T_{y'}M'\times T_{x'}M' \ar[r] & N_{y'}O' \ar[r]  & 0  }$$
We conclude that $G \cong (M\times M)\times_{M'\times M'}G'$ and hence $\phi$ is fully faithful.
\end{proof}

By using this characterization we can easily get the following saturation properties of the class of Morita maps.

\begin{corollary}\label{wesatura}\

\begin{itemize}
 \item If two maps are isomorphic and one of them is Morita, then so does the other (cf. \ref{behavior.isomorphic.maps}).
 \item In a commutative triangle of Lie groupoid maps, if two out of the three maps are Morita, then the third also is.
 \item If $\phi$ is such that there exist $\psi_1,\psi_2$ such that $\phi\psi_1$ and $\psi_2\phi$ are Morita, then $\phi$ is Morita as well.
\end{itemize}
\end{corollary}

% If $\phi:(G\toto M)\to(G'\toto M')$ is fully faithful then   is injective and 
% In fact, given $x,y\in M$, the following square is a good pullback (cf. \ref{diffanch})
% $$\xymatrix{
% G(x,y) \ar[r] \ar[d] \ar@{}[dr]|(.7){\lefthalfcap} & 
% G \ar[d]^\rho \\%\ar@{}[dr]|(.7){\lefthalfcap} & G' \ar[d]^{\rho_{G'}}\\
% x\times y \ar[r] & M\times M }$$%& M'\times M'}$$
% and it follows that $\phi$ induces diffeomorphisms .

%%%%%%%%%%%%%%%%%%%%%%%%%%%%%%%%%%

\subsection{Homotopy pullbacks}

The pullback of two Lie groupoid maps $\phi_1:(G_1\toto M_1)\to( G\toto M)$ and $\phi_2:(G_2\toto M_2)\to(G\toto M)$, if it exists, consists of the Lie groupoid whose objects and arrows
$$G_1\times_G G_2\toto M_1\times_M M_2$$
are the corresponding pullbacks of manifolds, and whose structural maps are
induced by those of $G_1$ and $G_2$.
We say that $\phi_1$ and $\phi_2$ are {\bf transverse} if they are so on objects and on arrows.
%As it happens with manifolds, transversality insures the existence of the pullback.

\begin{proposition}\label{transversality.groupoids}
If $\phi_1$ and $\phi_2$ are transverse then their pullback exists.
\end{proposition}

\begin{proof}
The pullbacks of manifolds $G'=G_1\times_G G_2$ and $M'=M_1\times_M M_2$ exist and are good (cf. \ref{pullback of manifolds}). The source, target and unit maps of $G_1$ and $G_2$ induce new maps $s',t':G'\to M'$ and $u':M'\to G'$ by the obvious formulas.
The key point is to show that $s':G'\to M'$ is a submersion.
Since $s'u'=\id_{M'}$ the differential of $s'$ is surjective near $u'(M')$.
We just need to prove that the dimension of
$$\ker d_{g'}s' = (T_{g'_1}G_1\times_{T_g G} T_{g'_2}G_2)\cap \ker d_{g'}(s_1,s_2)\subset T_{g'_1}G_1\times T_{g'_2}G_2$$
does not depend on $g'=(g_1',g_2')$.
Given $(y'_1,y'_2)\xfrom{(h'_1,h'_2)} (x'_1,x'_2)$ in $G'$, we have a diffeomorphism
$G_1(-,y'_1)\times G_2(-,y'_2)\cong G_1(-,x'_1)\times G_2(-,x'_2)$ defined by right multiplication by $h'$.
It is easy to see that its differential maps $\ker d_{g'}s'$ inside $\ker d_{g'h'}s'$. Since $g'$ and $h'$ are arbitrary we conclude that $\ker d_{g'}s'$ has constant dimension and hence $s'$ is a submersion. The rest is routine.
\end{proof}

\begin{remark}
A summary in pullbacks of Lie groupoids and algebroids will appear in \cite{bcd}.
Result \ref{transversality.groupoids} is stated in \cite[\S 5.3]{mm} without a proof.
In \cite[Prop. 2.4.14]{mkbook} it is proved that the pullback between a {\it fibration} and any map exists. This can be seen as a corollary of \ref{transversality.groupoids}, for it is immediate that a fibration is a surjective submersion on objects and arrows, and then it is transverse to any other map. 
\end{remark}

Homotopy pullbacks of Lie groupoids are an alternative for usual pullbacks that take consideration of the isomorphisms between maps. They play a relevant role in defining and composing generalized 	map of Lie groupoids. Next we provide their definitions and basic properties.  We suggest \cite{mm} as an alternative reference, where homotopy pullbacks appear with the name of {\it weak fibred products}.

\medskip

Given $\phi_1$ and $\phi_2$ as above, their {\bf homotopy pullback} $G_1\tilde\times_G G_2$  is defined as the pullback between the Lie groupoid maps
$(s,t):G^I \to G\times G$ and $\phi_1\times\phi_2:G_1\times G_2 \to G\times G$.
Its objects are triples $(x_1,\phi_1(x_1)\xfrom g\phi_2(x_2),x_2)$, and its arrows are triples $(k_1,k_2,k_3)$ as below.
$$\xymatrix{ x_1 \ar[d]_{k_1} & \phi_1(x_1) \ar[d]_{\phi_1(k_1)} & \phi_2(x_2) \ar[l]_{g} \ar[d]^{\phi_2(k_3)} & x_2 \ar[d]^{k_3} \\ y_1 & \phi_1(y_1)  & \phi_2(y_2) \ar[l]_{h} \ar[lu]_{k_2}
 & y_2}$$

% square
Since the groupoid of arrows $G^I$ classifies isomorphisms of maps, the homotopy pullback $G_1\tilde\times_G G_2$ fits into a universal commutative square up to isomorphism.
$$\begin{matrix}\xymatrix{G_1\tilde\times_G G_2 \ar[d] \ar[r] \ar@{}[dr]|(.7){\lefthalfcap} & G_1\times G_2 \ar[d]^{\phi_1\times\phi_2}\\G^I \ar[r]_{(s,t)}  & G\times G}\end{matrix}
\quad \longleftrightarrow \quad
\begin{matrix}\xymatrix{ G_1\tilde\times_G G_2 \ar[d]_{\tilde\phi_2} \ar[r]^{\tilde\phi_1} \ar@{}[dr]|(.7){\tilde\lefthalfcap} & G_2 \ar[d]^{\phi_2}\\G_1 \ar[r]_{\phi_1}  & G}\end{matrix}$$
The universal property of the pullback translates into the following.

\begin{remark}\label{universal.property.homotopy.pullback}
Given  $\psi_1:H\to G_1$, $\psi_2:H\to G_2$ and an isomorphism $\phi_1\psi_1\cong\phi_2\psi_2$, there is a unique map $\psi:H\to G_1\tilde\times_G G_2$ such that $\psi_1=\tilde\phi_2 \psi$, $\psi_2=\tilde\phi_1\psi$ and the isomorphism is naturally induced by $\psi$.
\end{remark}
In particular we have that $G_1\tilde\times_G G_2$ is a pullback in 
$\{\text{Lie Groupoids}\}/_{\simeq}$, the category of groupoids and isomorphism classes of maps.
We say that $\tilde\phi_1$, $\tilde\phi_2$ are the {\bf base-changes} of $\phi_1$, $\phi_2$ respectively.

\medskip

% Regarding the existence of homotopy pullbacks, it is enough to ask for the maps $\phi_1\times\phi_2$ and $(s_G,t_G)$ to be transverse (cf. \ref{transversality.groupoids}).
% Furthermore, in this case the transversality on arrows can be derived from the transversality on objects.

%  $$\begin{matrix}
%  \xymatrix{(M_1\times M_2)\underset{M\times M}\times G \ar[d] \ar[r] \ar@{}[dr]|(.7){\lefthalfcap} &  G\ar[d]^{\rho} \\M_1\times M_2 \ar[r]_{\phi_1\times\phi_2}  & M\times M}   
%    \end{matrix} \qquad \begin{matrix}
%  \xymatrix{(G_1\times G_2)\underset{M\times M}\times G \ar[d] \ar[r] \ar@{}[dr]|(.7){\lefthalfcap} &  G\ar[d]^{\rho} \\ G_1\times G_2 \ar[r]_{\phi_1 s\times\phi_2 t}  & M\times M}  \end{matrix}$$
We discuss finally the behavior of homotopy pullbacks with respect to Morita maps.
We say that $\phi:(G\toto M)\to (G'\toto M')$ is {\bf surjective} if $\phi:M\to M'$ is a surjective submersion. This implies essentially surjective, as it follows from the composition
$$M \cong M'\times_{M'} M \xto{(u,\id)} G'\times_{M'} M \xto{t\pr_1} M'$$
A {\bf surjective equivalence} is a map both surjective and fully faithful. In a surjective equivalence the map on arrows $G\to G'$ is also a surjective submersion.

\begin{proposition}\label{pbequiva}
Given $\phi_1$ and $\phi_2$ as above, if $\phi_1$ is Morita then the homotopy pullback $G_1\tilde\times_G G_2$ exists and the base-change $\tilde\phi_1$ is a surjective equivalence.
\end{proposition}

\begin{proof}[Sketch of proof]
The differential of $\phi_1$ is surjective in the direction normal to the orbits (cf. \ref{charequi}). Then the map $\phi^{ob}_1\times\phi^{ob}_2:M_1\times M_2\to M\times M$ is transverse to the anchor $\rho_G:G\to M\times M$ (cf. \ref{diffanch}). It follows from here that $\phi_1\times\phi_2$ is transverse to $(s,t)$ and hence the homotopy pullback 
$G_1\tilde\times_G G_2$ exists (cf. \ref{transversality.groupoids}).

The map $\tilde \phi_1$ is surjective because in the pullback of manifolds
$$\xymatrix{M_1\times_M G\times_M M_2 \ar[d]_{\rm pr} \ar[r]^(.67){\tilde\phi_1} \ar@{}[dr]|(.7){\lefthalfcap} & M_2 \ar[d]^{\phi_2}\\
M_1\times_M G \ar[r]_{s\pr_2}  & M}$$
the bottom arrow is a surjective submersion and therefore the upper one also is.
It remains to prove that $\tilde\phi_1$ is fully faithful. The pullback manifold between the anchor $\rho_2$ and the map $\tilde\phi_1\times\tilde\phi_1$ is 
$$\big((M_1\times_M G\times_M M_2)\times(M_1\times_M G\times_M M_2)\big)\times_{M_2}G_2 $$
% %\qquad
By rearranging the coordinates and using the multiplication of $G$ this manifold is diffeomorphic to 
$$\big((M_1\times M_1)\times_M G\big)\times_M G\times_M G_2$$
%$$\hspace{30pt}(x_1,g,x_2,y_1,g',y_2,g_2) \hspace{50pt}\mapsto  \hspace{20pt}(x,y,g'\phi(g_2)g,g,g_2)$$
Since $\phi_1$ is fully faithful we can replace $(M_1\times M_1)\times_M G \cong G_1$ and conclude that $\tilde\phi_1$ is fully faithful as well.
\end{proof}

% Here $\pi,\pi'$ are called the {\bf change of base} of $\phi,\psi$ respectively. The diagram above is universal in the following sense: for each groupoid $Z$, maps $\alpha:Z\to G$, $\alpha':Z\to G'$ and isomorphism of maps $\theta:\alpha\cong\alpha':Z\to H$, there is a unique map $\eta:Z\to G\times^w_H G'$ such that $\pi j=\alpha$, $\pi' j=\alpha'$, and $\theta$ is induced by the canonical isomorphism. This universal property follows directly from the definition as a strict pullback. Thus the weak pullback is the actual pullback when working with Lie groupoids and isomorphisms classes of maps. The universal property of the weak pullback can be extended to a 2-categorical one, on which $\pi j\cong\alpha$ and $\pi' j\cong\alpha'$ are settled isomorphisms, but this 2-categorical universal property will only define the weak pullback up to a categorical equivalence.

%%%%%%%%%%%%%%%%%%%%%%%%%%%%%%%%%%

\subsection{Equivalent groupoids and generalized maps}

Two Lie groupoids $G,G'$ are {\bf equivalent}, notation $G\sim G'$, if there is a third groupoid $H$ and Morita maps $H\xto\sim G$, $H\xto\sim G'$.
Equivalent groupoids have homeomorphic orbit spaces $M/G\cong M'/G'$, and for every pair of points $x\in M$, $x'\in M'$, whose classes are related by this homeomorphism, the corresponding normal representations $(G_x\action N_xO)\cong (G'_{x'}\action N_{x'}O')$ are isomorphic (cf. \ref{charequi}).
%$$ G\sim G' \quad \leadsto \quad\begin{cases} M/G\cong M'/G' \\ (G_x\action N_xO)\cong (G'_{x'}\action N_{x'}O')\end{cases}$$

\begin{example}\

\begin{itemize}
  \item A Lie groupoid $G\toto M$ is equivalent to a manifold $N\toto N$ if and only if $G$ is a submersion groupoid with quotient $M/G\cong N$.
  \item A Lie groupoid $G\toto M$ is equivalent to a Lie group $H\toto\ast$ if and only if $G$ is transitive and the isotropy at a point is $H$.
  \item Given $O$ an orbifold and $\U_1$, $\U_2$ two numerable atlases, by picking a common refinement $\U$ we can see that the induced groupoids $G_1\toto M_1$ and $G_2\toto M_2$ are equivalent (cf. \ref{example.weak.equivalence}). Thus, to an orbifold $O$ we can associate a Lie groupoid $G(O)=(G\toto M)$, which is determined up to canonical equivalence. 
\end{itemize} 
\end{example}

A pair of Morita maps $H\xto\sim G$, $H\xto\sim G'$ is an example of a generalized map. Given $G$, $G'$ Lie groupoids, a {\bf generalized map} $\psi/\phi:G\dasharrow G'$ is defined by two maps
$\phi: H\xto\sim G$ and $\psi: H\to G'$ where the first is Morita.
$$\xymatrix{G & H \ar[l]^\sim_{\phi} \ar[r]^\psi & G'}$$
Two pairs define the same generalized map, $\psi_1/\phi_1 =\psi_2/\phi_2$, if there is a third pair $\psi_3/\phi_3$ and they all fit into a diagram commutative up to isomorphisms.
$$\xymatrix@R=10pt{
 G \ar@{=}[d]& H_1 \ar[l]^\sim_{\phi_1} \ar[r]^{\psi_1} & G' \ar@{=}[d]\\
G \ar@{=}[d]& H_3 \ar[u] \ar[d] \ar[r]^{\psi_3} \ar[l]^\sim_{\phi_3} & G' \ar@{=}[d]\\
G & H_2 \ar[l]^\sim_{\phi_2} \ar[r]^{\psi_2} & G'}$$
This is an equivalence relation on pairs $(\phi,\psi)$ as it can be proved by using homotopy pullbacks (cf. \ref{pbequiva}).
%Note that $\psi/\phi$ only depends on the isomorphism classes of $\phi$ and $\psi$.
We denote by $H^1(G,G')$ the set of generalized maps $G \dashto G'$.
$$H^1(G,G')=\{\psi/\phi:G\dashto G'\}$$
It is easy to see that, after the identifications, this is in fact a set and not a proper class.

\begin{example}\label{example.generalize.maps}\

\begin{itemize}
\item
Generalized maps between manifolds $M\dashto M'$ are the same as usual smooth maps. 
$$H^1(M,M')= {\rm Maps}(M,M')$$
\item
Generalized maps between Lie groups $G\dashto G'$ are usual maps modulo inner automorphisms of $G'$. 
$$H^1(G,G')= {\rm Maps}(G,G')/_\cong$$
\item
Given a principal groupoid-bundle $G\action P\to N$, we can construct a generalized map $N\dashto G$ as follows (cf. \ref{section principal bundles}).
$$(N\toto N)\xfrom\sim(P\times_N P\toto P)\overset{\xi}{\cong}(G\times_MP\toto P)\to(G\toto M)$$
This construction sets a 1-1 correspondence  (cf. \ref{bibundles&generalizedmaps}).
$$H^1(N,G)\cong\{\text {principal $G$-bundles over $N$}\}$$
\item
We can identify orbifold maps $O\to O'$ and generalized maps between the induced groupoids (cf. \cite{mm}).
$$H^1(G(O),G(O'))= {\rm Maps}(O,O')$$
\end{itemize}
\end{example}

% fractions and surjective equivalences
\begin{remark}
By playing with homotopy pullbacks (cf. \ref{pbequiva}) it is easy to see that every generalized map can be presented as a fraction $\psi/\phi$ where $\phi:H\xto\sim G$ is a surjective equivalence (see eg. \cite{mm}).
% In fact, it is enough to consider any representative $\psi/\phi$, where $\phi:H\xto\sim G$ and $\psi:H\to G'$, and replace it by the equivalent one (cf. \ref{pbequiva})
% $$\xymatrix{G  & %\ar@{=}[d]& 
% G\tilde\times_G H \ar[l]_\sim  &
% (G\tilde\times_G H)\tilde\times_H (H\tilde\times_{G'}G') \ar[l]_(.65)\sim \ar[r]  &
% H\tilde\times_{G'}G' \ar[r] &
% G' }$$%\ar@{=}[d] \\ G & & H \ar[ll]^\sim_\phi \ar[rr]^\psi & & G'}$$
%where the homotopy pullback squares are constructed out of $\phi$, $\psi$, $\id_G$ and $\id_{G'}$.
The same argument shows that an equivalence $G\sim G'$ can always be realized by two surjective equivalences $H\xto\sim G$, $H\xto\sim G'$.
\end{remark}

%Here the three involved squares are homotopy pullbacks.
% $$\xymatrix{ (G\tilde\times_G H)\tilde\times_H (H\tilde\times_{G'}G') \ar[d]_\sim \ar[r] & 
% (H\tilde\times_{G'}G') \ar[r] \ar[d]_\sim & G' \ar@{=}[d]\\
% (G\tilde\times_G H) \ar[d]_\sim \ar[r] & H \ar[d]_\sim \ar[r] &  G'\\
%  G \ar@{=}[r]& G}$$
%The same argument shows that an invertible generalized map $\psi/\phi:G\dashto G'$ can be represented by a fraction where both $\phi,\psi$ are surjective equivalences. 

% cocycle description

Generalized maps also admit a cocycle description.
Given a Lie groupoid $G\toto M$ and a numerable open covering $\U=\{U_i\}_i$ of $M$, denote $M_\U=\coprod_i U_i$ and 
% $G_\U= G \times_{M\times M}(M_\U\times M_\U)= \coprod_{i,j}G(U_i,U_j)$.
$G_\U= \coprod_{i,j}G(U_i,U_j)$.
The structure of $G$ induce a new Lie groupoid $G_\U\toto M_\U$ and a surjective equivalence
$$\phi_\U:(G_\U \toto M_\U)\xto\sim (G\toto M)$$
If $\U$, $\U'$ are open coverings and $\U$ refines $\U'$, then $\phi_\U$ clearly factors through $\phi_{\U'}$, and two such factorizations must be isomorphic.

\begin{proposition}\label{cocycle}
Every generalized map $G\dashto G'$ can be realized as a fraction $\psi/\phi_\U$ for some numerable open covering $\U$ of $M$, and two fractions agree if, when expressed over the same covering $\U$, their numerators are isomorphic. 
$$H^1(G, G') = \varinjlim_{\U} {\rm\ Maps}(G_{\U},G')/_\cong$$
\end{proposition}
\begin{proof}
Given $\psi/\phi:G\dashto G'$ with $\phi: (H\toto N)\xto\sim (G\toto M)$ a surjective equivalence, the map on objects $\phi:N\to M$ is a surjective submersion and therefore it admits local sections over some open covering $\U$ of $M$. A choice of such sections provides a factorization
$$\xymatrix{ G_\U \ar[rr] \ar[rd]_{q_\U} & & H \ar[dl]^\phi \\ & G &}$$
and the resulting map $G_\U \xto\sim H$ is Morita. It is determined up to isomorphism by \ref{categorical.monomorphism}. This also proves the second statement.
\end{proof}

%%%%%%%%%%%%%%%%%%%%%%%%%%%%%%%

\subsection{Bibundles as generalized maps}

There is another approach to equivalences and generalized maps via principal bibundles. 
% defining bibundle and principal bibundles
Let $G, G'$ be Lie groupoids. A left and right actions $G\action P \raction G'$ define a {\bf bibundle} if they commute and the moment map of one is invariant for the other. We depict the situation by
$$\begin{matrix}
G & \action & P & \curvearrowleft & G' \\
\downdownarrows & \swarrow & & \searrow & \downdownarrows\\
M & & & & M'
\end{matrix}$$
We call $G\action P\to M'$ and $M\from P \raction G'$ the left and right {\bf underlying bundles}.
A bibundle is {\bf principal} if both underlying bundles are so. 
A bundle is left (resp. right) principal if only the right (resp. left) underlying bundle is so.

\begin{example}\

\begin{itemize}
 \item A left principal bibundle $G\action P \raction N$ between a Lie groupoid $G$ and a manifold $N$ is the same as a principal bundle $G\action P\to N$ (cf. \ref{section principal bundles}). 
 \item The left and right multiplications $G\action G \raction G$, $g\cdot g \cdot g'=ghg'$, constitute a principal bibundle.
 \item Given $G$ and $A\subset M$ , we denote by ${\langle A\rangle}$ its saturation. If the restrictions $G_A$, $G_{\langle A\rangle}$ are well-defined, then previous example restricts to a principal bibundle $G_{\langle A\rangle}\action G(-,A)\raction G_A$. In particular when $G$ is transitive and $x\in M$, we have a principal bibundle $G\action G(-,x)\raction G_x$.
\end{itemize}
\end{example}

%%%%%%%%%%

% rigidity of principal bibundles
A free proper action $G\action P$ leads to a principal bibundle as follows. The action gives a principal bundle $G\action P\to P/G$ (cf. \ref{charpbun}) and hence a gauge groupoid $P\times_{M}^{G} P\toto P/{G}$ (cf. \ref{gauge}). An arrow of this gauge groupoid is denoted by $[a',a'']$, with $a',a''\in P$ in the same fiber of the moment map.
The gauge groupoid acts over $P$ on the right, $P\raction P\times_{M}^{G}P$, by the formula
$$a \cdot[a',a'']= g\cdot a' \qquad \iff \qquad  g\cdot a''=a$$
This action is free, proper and compatible with that of $G$, yielding a principal bibundle
$$\begin{matrix}
G & \action & P & \curvearrowleft & P\times_{M}^{G}P \\
\downdownarrows & \swarrow & & \searrow & \downdownarrows\\
M & & & & P/G
\end{matrix}$$
It turns out that every principal bibundle arises in this way.
%In other words, we can recover a principal bibundle $G\action P\raction G'$ from the action $G\action P$.

\begin{proposition}\label{gauge&bibundles}
Given a principal bibundle $G\action P \raction G'$, there is a canonical isomorphism $G'\cong (P\times_{M}^{G} P)$ compatible with the actions.
\end{proposition}

\begin{proof}
Since $M\from P\raction G'$ is also a principal bundle we have an isomorphism
$$(P\times_{M} P\toto P)\cong(P\times_{M'} G' \toto P)$$
that identifies the orbits of the actions $G\action P$ and $G\action P\times_MP$ with the fibers of the action map $(P\times_{M'}G'\toto P)\to (G'\toto M')$. The result now follows.
\end{proof}
% $$\begin{matrix}
% G & \action & P & \curvearrowleft & P\times_M^G P \\
% \downdownarrows & \swarrow & & \searrow & \downdownarrows\\
% M & & & & P/G
% \end{matrix}$$

% In fact, the gauge construction (cf. \ref{gauge})
% of the underlying bundle $G\action P\to M'$ naturally acts on $P$ on the right, by the formula 
% It is not hard to check that this action is smooth, free, and proper for it defines the same relation as the submersion $P\to M$. Thus there is a principal bibundle

%%%%%%%%%%

\begin{theorem}\label{bibundles&generalizedmaps}
There is a 1-1 correspondence between generalized maps and isomorphism classes of right principal bibundles.
$$H^1(G,G')\cong\{\text{right principal bibundles }G\action P\raction G'\}$$
Under this correspondence, equivalences corresponds to principal bibundles.
\end{theorem}

\begin{proof}
Given a bibundle $G\action P\raction G'$, we can construct an action groupoid of the simultaneous action
$$G\ltimes P \rtimes G' = (G\times_{M} P\times_{M'} G' \toto P)$$
with source $(g,a,g')\mapsto a$, target $(g,a,g')\mapsto gag'$, and unit, multiplication and inverses induced by those of $G$ and $G'$.
There are obvious projections $\pi_1:G\ltimes P \rtimes G' \to G$, $\pi_2:G\ltimes P \rtimes G'\to G'$, and it is easy to see that $\pi_1$ (resp. $\pi_2$) is Morita if and only if the bibundle is right (resp. left) principal.
$$(G\action P\raction G') \quad \overset{\alpha}{\mapsto} \quad (G \from G\ltimes P \rtimes G' \to G')$$

On the other hand, given a fraction of groupoid maps 
$(G\underset{\sim}{\xfrom\phi} H \xto\psi G')$, we can construct a bibundle as follows.
Consider the following groupoid action.
$$(H\toto N)\action (G\times_{M} N\times_{M'} G'\xto{\pr_2} N) \qquad h\cdot(g,a,g')=(g\phi(h)^{-1},t(h),\psi(h)g')$$
This action is free and proper because $\phi$ is fully faithful. The groupoids $G,G'$ act on the quotient manifold $(G\times_{M} N\times_{M'} G')/H$ by
$\tilde g\cdot[g,n,g']\cdot\tilde g'=[\tilde g g,n,g'\tilde g']$.
This bibundle is right principal, and it is left principal if and only if $\psi$ is Morita as well.
$$(G\action\frac{(G\times_{M} N\times_{M'} G')}{H}\raction G') \quad \overset{\beta}{\leftmapsto}\quad
( G \from H \to G')$$

There is a natural isomorphism $\beta\alpha(G\action P\raction G')\xto\cong (G\action P\raction G')$ defined by $[g,a,g']\mapsto gag'$.
Regarding the other composition there is a natural map 
$$(H\toto N)\to \alpha\beta(H\toto N) \qquad n\mapsto [1,n,1] \quad h\mapsto (\phi(h)^{-1},\psi(h))$$ that establish an identity of generalized maps.
$$\xymatrix{
G \ar@{=}[d] & H \ar[l]_\sim \ar[r] \ar[d] & G' \ar@{=}[d] \\
G & G\ltimes\frac{(G\times_M N \times_{M'}G')}{H}\rtimes G' \ar[l]_(.7)\sim \ar[r] & G'
}$$
\end{proof}

Under the above correspondence, a right principal bibundle $G\action P\raction G'$ is associated to a map $f:G\to G'$ if and only if its right underlying bundle is trivial. In fact, if we denote by $P(\psi/\phi)$ the bibundle associated to the fraction $\psi/\phi$, we have
$P = P(f/1) = \frac{(G\times_M M \times_{M'}G')}{G} = M\times_{M'} G'$. In particular, a principal bibundle corresponds to a Morita map if and only if it admits a global section.
% $$\xymatrix{
% & G\ltimes P\rtimes G' \ar[dl]_{\pi_1}^\sim \ar[dr]^{\pi_2} & \\ G \ar[rr]_f & & G'
% }$$

\begin{remark} % principal bibundles are principal bundles
A right principal bibundle $G\action P\raction G'$ can be thought of as a principal $G'$-bundle with base $G$. In fact, the moment map of the let action 
$$(G\times_M P \toto P)\to (G\toto M)$$
is a compatible diagram of Lie groupoids and principal $G'$-bundles, and every such diagram comes from a right principal bundle (cf. \ref{actiactm}, \ref{mapspbun}).
From this point of view, proposition \ref{bibundles&generalizedmaps} is saying that each $G$-bundle over $G'$ is the pullback along a unique generalized map of the universal bundle $G\action G^I\xto s G$.
\end{remark}

\subsection{Differentiable stacks}
\label{sectionss}

Every Lie groupoid has an underlying differentiable stack. 
A generalized map between Lie groupoids is a map between their differentiable stack, and two Lie groupoids are equivalent if and only if their stacks are isomorphic.
In this section we formalize these ideas.

\bigskip

Given two generalized maps $\psi/\phi:G\dashto G'$, $\psi'/\phi':G'\dashto G''$, their {\bf composition} is defined as $\psi'\psi''/\phi''\phi$, where $\phi'':K\xto\sim H$ and $\psi':K\to H'$ are such that the square below commutes up to isomorphism.
$$\xymatrix@R=10pt{ & & K \ar[dl]^\sim_{\phi''} \ar[dr]^{\psi''} & & \\
 & H \ar[dl]^\sim_\phi \ar[dr]^{\psi} & & H' \ar[dl]^\sim_{\phi'} \ar[dr]^{\psi'} & \\ G & & G' & & G''}$$
We can take as $K$ the homotopy pullback $H\tilde\times_{G'} H'$, and any other choice will lead to an equivalent fraction (cf. \ref{universal.property.homotopy.pullback}, \ref{pbequiva}).
With this composition we get a well-defined category of Lie groupoids and generalized maps.
Given a Lie groupoid $G\toto M$, we define its underlying {\bf differentiable stack} $[M/G]$ as the object it defines in this category.
$$\{\text{Differentiable stacks}\} = \{\text{Lie Groupoids}\}/_{\sim}$$
Maps of differentiable stacks $[M/G] \to [M'/G']$ are, by definition, generalized maps of Lie groupoids
$G\dashto G'$. %We denote the set of such maps by $H^1(G,G')$.
% $${\rm Maps}([M/G],M'//G')= \{G\dashto G'\}= H^1(G,G')$$
%$$ [M/G] \to M'//G' \quad \leadsto \quad\begin{cases} M/G\to M'/G' \\ (G_x\action N_xO)\to (G'_{x'}\action N_{x'}O')\end{cases}$$
Such a map induces a map between the orbit spaces and maps between the normal representations (cf. \ref{charequi}).
We can think of $[M/G]$ as the topological space $M/G$ with some smooth data attached.

\begin{example}
We can identify manifolds with their unit groupoids and their underlying differentiable stacks. This way we can see the category of differentiable stacks as an extension of that of manifolds (cf. \ref{example.generalize.maps}).
The same happens with orbifolds. We can actually define orbifolds as the underlying differentiable stacks to certain Lie groupoids (proper with finite isotropy).
\end{example}

\begin{remark}
The construction of the category of differentiable stacks from that of Lie groupoids can be framed into the theory of localization and calculus of fractions (see eg. \cite[\S 7.1]{ks}).
The category $\{\text{differentiable stacks}\}$ is obtained from
$\{\text{Lie Groupoids}\}/_{\simeq}$ by formally inverting the Morita maps.
The description of the maps as fractions $\psi/\phi$ is a consequence of the general theory, once it is proved that the class of maps we are inverting is a left multiplicative system (cf. \ref{wesatura} and \ref{pbequiva}).
\end{remark}

According to \ref{wesatura} Morita maps are {\it saturated} (cf. \cite[\S7.1]{ks}), and thus a generalized map $\psi/\phi$ is invertible if and only if $\psi$ is Morita. In other words, two Lie groupoids are equivalent $G\sim G'$ if and only if their differentiable stacks are isomorphic $[M/G]\cong [M'/G']$. Of course, this can also be proved directly.

% An orbit $O\subset M$ defines a point of the singular space, and the representation $G_x\action N_x$ can be thought of as the tangent space of $[M/G]$ at $[x]$.

\bigskip

% \begin{proof}
% It is easy to see that a generalized map $\psi/\phi$ defined by two weak equivalences is invertible with inverse $\phi/\psi$. So let us prove the converse. 
% 
% A generalized map $\psi/\phi$ is an identity if and only if $\phi\cong\psi$. In particular, $\psi$ has to be a weak equivalence. Now, if a generalized map $\psi/\phi$ is invertible, then there are maps $\phi_1$, $\phi_2$ such that $\phi_1\psi$ and $\psi\phi_2$ are weak equivalences, and thus by \ref{wesatura} $\psi$ has to be a weak equivalence as well.
% \end{proof}
% 
 
Given $G\toto M$ a Lie groupoid, the canonical inclusion $(M\toto M)\to(G\toto M)$
induces a map $\pi:M\to [M/G]$ of differentiable stacks, which we call the {\bf presentation} of $[M/G]$ induced by $G\toto M$.
$$G\toto M \quad \leadsto \quad M \to [M/G] $$ %\begin{matrix}\xymatrix@R=10pt{M \ar[d] \\ [M/G]}\end{matrix}$$
It turns out that a Lie groupoid is essentially encoded in this presentation.

\begin{theorem}\label{presentations}
There is a 1-1 correspondence between isomorphism classes of maps $(G\toto M)\to(G'\toto M')$ and commutative squares in $\{\text{Differentiable Stacks}\}$ between the induced presentations.
$$(G\toto M)\xto\theta(G'\toto  M') \quad \longleftrightarrow \quad
\begin{matrix}\xymatrix@R=10pt{ M\ar[r]^f \ar[d] & M'\ar[d] \\ [M/G] \ar[r]^{\psi/\phi} & [M'/G']}\end{matrix}$$
\end{theorem}

\begin{proof}
It is clear that a map $\theta$ induces one of these commutative squares of differentiable stacks. 
Let us prove the converse.

Let $f:M\to M'$ and $\psi/\phi:[M/G]\to [M'/G']$ be such that $\pi'f=(\psi/\phi) \pi$.
We can assume that $\phi$ is a surjective equivalence. 
Write $j:(K\toto N)\subset (H\toto N)$ for the {\bf kernel} of $\phi$, say the Lie groupoid of arrows that are mapped by $\phi$ into identities.
In the following diagram of Lie groupoids
$$\xymatrix{
M \ar[d]_\pi \ar@/^1pc/[rr]^f & K\ar[l]_\sim^{\phi|_K} \ar[d]^j & M'\ar[d]^{\pi'} \\
G & H \ar[l]_\sim^\phi \ar[r]_\psi & G'}$$
the left square commutes on the nose, hence $(\psi/\phi) \pi = (\psi j)/ \phi|_K$.
From the identity $(\psi j)/ \phi|_K = \pi'f$ we deduce that there is an isomorphism of Lie groupoid maps 
$$\alpha:\psi j \cong \pi' f \phi|_K:K \to G' \qquad  f(\phi(n))\xfrom{\alpha(n)} \psi(n) \ \text{ for all } n\in N$$
Now we use $\alpha$ to {\it twist} the map $\psi$. More precisely, we define $\tilde \psi:H\to G'$ by 
$$\tilde \psi( n' \xfrom h n) = (f\phi(n') \xfrom{\alpha(n')\psi(h)\alpha(n)^{-1}}f\phi(n))$$
The map $\alpha$ gives an isomorphism $\tilde\psi\cong\psi$, and $\tilde\psi$ is constant along the fibers of $\phi:H\to G$, hence it induces a map $\theta:G\to G'$
as required.
\end{proof}

%Previous result can be strengthened by considering the isomorphisms between maps rather than identifying isomorphic maps as we do.

\begin{remark}\label{stack}
The category of differentiable stacks can alternatively be constructed by using stacks.
Roughly speaking, stacks are sheaves of groupoids,
they extend the notion of spaces,
and have proved to be a useful tool especially in moduli problems.
Within that framework, what we called a differentiable stack is just a stack over the category of manifolds which can be presented as a quotient of a manifold, and a Lie groupoid is one of such presentations. 
We refer to \cite{bx} and \cite{metzler} for an introduction to stacks in general and how to apply the theory in manifolds. See also \cite{gh}.
\end{remark}

\section{Proper groupoids}\label{sectprop}

This section deals with proper groupoids and the geometry of their differentiable stacks.
Along the several subsections we include: definitions and examples; properties of orbits and slices; a discussion on stability; Zung's Theorem; and an overview on linearization.

We use \cite{mm}, \cite{weinstein} as general references for this section, and we especially follow \cite{cs} for Zung's theorem and the linearization discussion. As we explained in the introduction, a completely new approach to the topic will be presented in \cite{riemannian}.

\subsection{Proper groupoids}

% definition & first consequences
A Lie groupoid $G\toto M$ is said to be {\bf proper} if its anchor 
$\rho=(t,s):G\to M\times M$
is a proper map (cf. \S \ref{sectionpropmaps}).
Equivalently, a groupoid is proper if given compact sets $K,K'\subset M$ the set of arrows between them $G(K,K')$ is compact as well.

Since a proper map is closed with compact fibers, in a proper groupoid the isotropy groups $G_x$ are compact, the relation $\rho(G)\subset M\times M$ is closed, and therefore the orbit space $M/G$ is Hausdorff.

% examples
\begin{example}\

\begin{itemize}
\item
Given $M$ a manifold, its unit groupoid $M\toto M$ and its pair groupoid $M\times M\toto M$ are proper. More generally, a submersion groupoid $M\times_N M\toto M$ is the same as a proper groupoid without isotropy (cf. \ref{anchinje}).
\item
A Lie group $G\toto \ast$ is proper if and only if it is compact.
More generally, a transitive groupoid is proper if and only if its isotropy at a point is compact. For instance, the general linear groupoid $GL(E)$ of a vector bundle $E$ is not proper, but the orthogonal groupoid $O(E)$ is so (cf. \ref{linear.groupoids}).
\item
By definition, an action $(G\toto M)\action (A\to M)$ is proper if the action groupoid $G\times_M A\toto A$ is so (cf. \ref{subsection.groupoid.actions}, see also \cite{dk} for the group case).
\item
The Lie groupoid arising from a covering of an orbifold is proper (cf. \cite{mm}).
\end{itemize}
\end{example}

% properness at a point
There is a local version for the notion of properness.
A Lie groupoid $G\toto M$ is {\bf proper at $x$} if its anchor map $\rho$ is proper at $(x,x)$.
A proper groupoid is proper at every point, but the converse is not true.

\begin{example}
Let $G\toto M$ be the groupoid without isotropy whose objects are the non-zero points in the plane, and whose orbits are the leaves of the foliation by horizontal lines. This groupoid is proper at every point, but it is not proper, for $M/G$ is not Hausdorff. 
\end{example}

%The characterization below is analog to that of proper actions .

\begin{proposition}\label{punctual}(Compare with \cite[2.5]{dk})
A groupoid $G\toto M$ is proper if and only if it is proper at every point and the orbit space $M/G$ is Hausdorff.
\end{proposition}

\begin{proof}
Let $G\toto M$ be such that $M/G$ is Hausdorff and the anchor $\rho$ is proper at $(x,x)$ for all $x$. We have to show that $\rho$ is proper at any point $(y,x)$.
Since $M/G$ is Hausdorff the relation $\rho(G)\subset M\times M$ is closed and the anchor is obviously proper at points $(y,x)\notin \rho(G)$.
On the other hand, if there is an arrow $y\xfrom g x$ in $G$, the translation by a bisection through $g$ shows that the anchor over $(y,x)$ behaves as over $(x,x)$, and thus the result.
\end{proof}

\begin{remark}\label{proper.neighborhood}
Given any Lie groupoid $G\toto M$, the points $U\subset M$ at which it is proper is open and saturated. It is open because of the local nature of properness (cf. \ref{openproper}) and it is saturated because of the argument with bisections used in the proof above.
It follows that a groupoid $G\toto M$ is proper at a point $x\in M$ if and only if there exists a saturated open $x\in V\subset M$ such that the restriction $G_V\toto V$ is proper: we can take $x\in U\subset M$ small so as to make $G_U\toto U$ proper, and then take $V$ as its saturation.
\end{remark}

% Morita invariance
Properness is invariant under equivalences, it is a property of the differentiable stack rather than the Lie groupoid itself.

\begin{proposition}
If two Lie groupoids are equivalent and one of them is proper, then so does the other.
\end{proposition}

\begin{proof}
Equivalent groupoids can always be linked by surjective equivalences.
Thus, let $\phi:(G\toto M)\to (G'\toto M')$ be a surjective equivalence.
Since $\phi$ is fully faithful the square
$$\xymatrix{ G \ar@{}[dr]|(.7){\lefthalfcap} \ar[r] \ar[d]_\rho & G' \ar[d]^{\rho'} \\ M\times M \ar[r] & M'\times M'}$$
is a pullback, and thus $\rho$ is a base-change of $\rho'$.
On the other hand, since $\phi$ is surjective, we can locally express $\rho'$ as a base-change of $\rho$ by using local sections of $M\times M\to M'\times M'$.
Since properness of maps is stable under base-change the result follows.
\end{proof}

Previous proposition admits a punctual version: If 
$[M/G]\to [M'/G']$ is an isomorphism of differentiable stacks mapping $[x]$ to $[x']$, then $G\toto M$ is proper at $x$ if and only if $G'\toto M'$ is proper at $x'$. This can be seen by restricting the equivalence to suitable saturated open subsets $V\subset M$ and $V'\subset M'$.

We say that a differentiable stack is {\bf separated} if it is associated to a proper groupoid.

%%%%%%%%%%%%%%%%%%%%%%%%%%%

\subsection{Orbits and slices}

%In this section we show that orbits of proper groupoids are closed embedded submanifolds, and we prove the existence of slices, which are small submanifolds transversal to the orbits, providing a parametrization of a neighborhood of the orbit in the quotient $[M/G]$.

% Orbits
Given $G\toto M$ a Lie groupoid and $x\in M$, the source-fiber $G(-,x)\subset G$ is an embedded submanifold, the isotropy acts $G(-,x)\raction G_x$ freely and properly, and the orbit $G(-,x)/G_x\cong O_x$, whose manifold structure is that of the quotient, is included as a submanifold $O_x\subset M$. It may be the case that the orbit is not embedded.

\begin{example}\label{kronecker}
The foliation on the torus $T=S^1\times S^1 = \R^2/\Z^2$ induced by the parallel lines on $\R^2$ of some irrational slope is called {\it Kronecker foliation}. We can define a Lie groupoid without isotropy, with an arrow between two points if they belong to the same leaf. The orbits on this Lie groupoid are exactly the leaves, which are not embedded submanifolds.
\end{example}

\begin{proposition}\label{embedded-orbit}
If $G\toto M$ is proper at $x$ then $O_x\subset M$ is closed and embedded.
\end{proposition}
\begin{proof}
By restricting to a neighborhood we can assume that $G$ is proper (cf. \ref{proper.neighborhood}).
The orbit is closed because it is a fiber of the quotient map $M\to M/G$ and $M/G$ is Hausdorff.
Write $\tilde O_x\subset M$ for the orbit endowed with the subspace topology, an consider the following topological pullback.
$$\xymatrix{
G(-,x) \ar[r] \ar[d] \ar@{}[dr]|(.7){\lefthalfcap} & G \ar[d]^\rho \\ \tilde O_x\times  x \ar[r] & M\times M }$$
Since the right map is proper, the left one is closed and hence it is a topological quotient.
We conclude that both the quotient and the subspace topologies on the orbit agree, namely $O_x=\tilde O_x$, and we are done.
\end{proof}

% slices
Given $G\toto M$ and $x\in M$, a {\bf slice} of $G$ at $x$ is an embedded submanifold 
$x\in S\subset M$ such that
(1) $S$ is transverse to the orbits, and
(2) $S$ intersects $O_x$ only at $x$.
This notion is close to those of slices for Lie group actions, and transverse sections to  foliations. However, note that a slice for an action groupoid need not to be a slice for the corresponding group action (cf. \cite{dk}).

Slices may not exist in general (see eg. \ref{kronecker}) but do exist for proper groupoids.

\begin{proposition}\label{exissli}
If $G\toto M$ is proper at $x$ then there is a slice $S$ at $x$.
%The restriction $G_S\toto S$ is a proper groupoid.
\end{proposition}

\begin{proof}
Since $G$ is proper at $x$ the orbit $O_x\subset M$ is an embedded submanifold (cf. \ref{embedded-orbit}). Then we can take a manifold chart
$$\phi:\R^p\times\R^q\to U\subset M$$
such that $\phi(0,0)=x$ and $\phi^{-1}(O_x\cap U)=\R^p\times \{0\}$.
Consider $S'=\phi(\{0\}\times\R^q)\subset M$, which is an embedded submanifold that intersects $O_x$ only at $x$.
We can take as a slice the open subset $S\subset S'$,
$$S=\{y | T_yO_y+T_yS'=T_yM\}$$
which is open for it is the locus on which some matrices have maximum rank.
%The restriction $G_S$ is a proper groupoid because its anchor is a base-change of that of $G$.
\end{proof}

As we explained in \ref{section.normal.representation}, the restriction of a Lie groupoid $G\toto M$ to a submanifold $A\subset M$ may not be well-defined, even if $A$ is embedded.
Next we use what we know on the differential of the anchor to show that we can restrict a Lie groupoid to a slice. 

\begin{proposition}
Given $G\toto M$ a Lie groupoid and $S$ a slice at $x$, the restriction $G_S\toto S$ is a well-defined Lie groupoid.
\end{proposition}

\begin{proof}
The map $S\times S\to M\times M$ is transverse to the anchor (cf. \ref{diffanch}), thus we have a good pullback
$$\xymatrix{G_S \ar[r] \ar[d]  \ar@{}[dr]|(.7){\lefthalfcap}& G \ar[d] \\ S\times S\ar[r] & M\times M}$$
and $G_S\subset G$ is an embedded submanifold.
In order to prove that $G_S\toto S$ with the induced structure is a Lie groupoid we just need to show that the source restricts to a submersion $s:G_S\to S$. Given $y\xfrom g x$ an arrow in $G_S$ and $v\in T_xS$, we look for a vector $\tilde v\in T_g G_S$ such that $ds(\tilde v)=v$. From the pullback of vector spaces
$$\xymatrix{T_gG_S \ar[r] \ar[d]  \ar@{}[dr]|(.7){\lefthalfcap}& T_gG \ar[d] \\ T_yS\times T_xS\ar[r] & T_yM\times T_xM}$$
we deduce that such a $\tilde v$ exists if and only if there is some $w\in T_yS$ with $(w,v)\in  {\rm Im}( d_g\rho)$. Since $S$ and $O_y$ are transverse, the composition $T_yS\to T_yM \to N_yO_y$ is surjective, from where there always exists such a $w$ (cf. \ref{diffanch}).
\end{proof}

Note that if $G\toto M$ is proper, then the restriction $G_S\toto S$ also is.

\begin{remark}\label{slice.chart}
A slice $S$ of $G$ at $x$ allows us to describe the geometry of the differentiable stack $[M/G]$ in a neighborhood of $[x]$.
In fact, the inclusion
$G_S\to G$
is fully faithful, and if we write $U$ for the saturation of $S$, then $U$ is open and we have an equivalence $G_S\sim G_U$, which is the same as an isomorphism between the differentiable stacks $[S/G_S]\cong [U/G_U]$.
\end{remark}

% A slice parametrizes a neighborhood of the orbit.

%%%%%%%%%%%%%%%%%%%%%

\subsection{Stable orbits}

Orbits of Lie groupoids play the role of the fibers of a submersion, the leaves of a foliations, or the orbits of group actions. We can export from there the notion of stability. 
Given $G\toto M$ a Lie groupoid, an orbit $O_x\subset M$ is called {\bf stable} if it admits arbitrary small invariant neighborhoods, namely for every open $U$, $O_x\subset U\subset M$, there exists a saturated open $V$ such that $O_x\subset V\subset U$.

\begin{example}
Let $M\times_N M\toto M$ be the Lie groupoid arising from a submersion $q:M\to N$.
An orbit $O_x$ is stable if and only if $q$ satisfies the tube principle at $q(x)$, i.e. if it is proper at $q(x)$ (cf. \ref{etheorem}).
In particular, a stable orbit has to be compact.
% For instance, if $q$ is the projection $\R^2\to\R$ then no orbit is stable neither compact;
% and if $q$ is the projection $(\R\times\{-1,1\})\setminus\{(0,1)\}\to\R$ then the orbit of $(0,-1)$ is compact but not stable.
\end{example}

As in previous example, compactness is always a necessary condition for stability.

\begin{proposition}
A stable orbit $O_x$ of a Lie groupoid $G\toto M$ is compact.
\end{proposition}

\begin{proof}
Let $d$ be a distance defining the topology of $M$.
If $O_x$ is not compact then it contains an infinite discrete set $\{x_n\}\subset O_x$. Let $B_n$ be the $d$-ball centered at $x$ of radius $1/n$, and let $\langle B_n\rangle$ be its saturation. For each $n$ we can take a point $y_n\in \langle B_n\rangle\setminus O_x$ such that $d(x_n,y_n)<1/n$. Then $M\setminus \{y_n\}$ is open and does not contain any invariant neighborhood, hence the orbit is not stable.
\end{proof}

Let $G\toto M$ be a Lie groupoid and let $x\in M$. We say that $G$ is {\bf s-proper} at $x$ if the source map $s:G\to M$ is proper at $x$. Note that s-proper at $x$ implies proper at $x$.

\begin{proposition}\label{stable2}
The following are equivalent:
\begin{enumerate}
\item $G$ is s-proper at $x$;
\item $G$ is proper at $x$ and $O_x$ is stable;
\item $G$ is s-locally trivial at $x$ and $G_x$ and $O_x$ are compact.
\end{enumerate}
\end{proposition}

\begin{proof}
The s-fiber $G(-,x)$ is compact if and only if the isotropy $G_x$ and the orbit $O_x$ are so, for these three fit into the isotropy bundle $G_x\action G(-,x)\to O_x$. This, together with Ehresmann theorem \ref{etheorem}, give the equivalence (1)$\iff$(3).

\smallskip

To prove (1)$\then$(2), suppose that $s:G\to M$ is proper at $x$. Given an open $U$ containing the orbit $O_x$, since $s^{-1}(O_x)\subset t^{-1}(U)$, by the tube principle (cf. \ref{charpatx}), we can take an open $V$, $x\in V\subset M$, such that $s^{-1}(V)\subset t^{-1}(U)$. Then $t(s^{-1}(V))$ is open invariant with $O_x\subset t(s^{-1}(V))\subset U$.  

\smallskip

Finally, to prove (2)$\then$(1), assume that $G\toto M$ is proper at $x$ and $O_x$ is stable, and thus compact.
Let $(y_n\xfrom{g_n} x_n)\subset G$ be such that $x_n$ converges to $x$. We may assume either that infinitely many $y_n$ belong to $O_x$, or that $y_n\notin O_x$ for any $n$.
In the first case, since the orbit is compact, there is a subsequence $(g_{n_k})$ whose source and target converge, and since $G$ is proper $(g_{n_k})$ has to have a convergent subsequence.
In the second case, $M\setminus\{y_n\}$ contains $O_x$ but does not contain any saturated open set. It follows that $M\setminus\{y_n\}$ is not open and then $\{y_n\}$ has to have a convergent subsequence and we can conclude as before.
\end{proof}

Particular cases and other versions of the previous result can be found in the literature (cf. \cite[4.4]{cs}, \cite[4.10]{cs}, \cite[3.3]{weinstein}).

%We remark the following interesting corollary: if $G\toto M$ is proper at $x$, the orbit $O_x$ is compact and the nearby s-fibers are connected, then $O_x$ is stable (cf. \ref{ehreconn}).

% Note that if $G\toto M$ is s-proper at $x$ then $s:G\to M$ is locally trivial at $x$ and the s-fiber $G(x,-)$ is compact. In view of the 
% 
% 
% 
% \item[1] $O_x$ is stable,
% \item[(ii)] $s:G\to M$ is proper at $x$, and
% \item[(iii)] $O_x$ is compact and $s:G\to M$ is locally trivial at $x$.
% \end{itemize}
% \end{proposition}
% \begin{proof}
% Recall that a submersion is proper at a point if and only if it is locally trivial with compact fiber (cf. \ref{etheorem}). Since $G(x,-)\to O_x$ is a principal bundle with compact group $G_x$, we have that $G(x,-)$ is compact if and only if $O_x$ is so. This proves (ii) $\iff$ (iii).
%  
% \end{proof}

%\subsection{Averaging}

%TO COMPLETE
%See the Crainic reference, J L Tu reference, and also the mentions of Giorgio.

\subsection{Zung's Theorem and the local structure of separated stacks}

% Introduction
Given a Lie groupoid $G\toto M$, we should think of the normal representation $G_x\action N_xO$ as the tangent space of the stack $[M/G]$ at $[x]$, for it encodes the infinitesimal linear information around the point.
$$T_{[x]}[M/G] \quad\cong\quad G_x\action N_xO$$
If $G\toto M$ is proper, it turns out that a neighborhood of $[x]$ in $[M/G]$ can be 
reconstructed as the differentiable stack of the action $G_x\action N_xO$.
Note that, by the existence of slices, it is enough to study neighborhoods of fixed points (cf. \ref{exissli}, \ref{slice.chart}).
In this sense we have Zung's Theorem, which is both a particular case of the Linearization Theorem \ref{linearization}, and the key step in proving it. %Thus, the singular space of a proper groupoid is locally modeled by the linear action of a compact group on euclidean space.

%It provides a local normal form for singular spaces of proper groupoids.

\begin{theorem}[Zung]\label{zung}
Let $G\toto M$ be a Lie groupoid and let $x\in M$ be a fixed point. If $G$ is proper at $x$ then there is an open $x\in U\subset M$ and an isomorphism
 between the restriction and the action groupoid of the normal representation at $x$.
$$(G_U\toto U)\cong (G_x\ltimes T_xM\toto T_xM)$$
\end{theorem}

%This suggest a chart-definition for singular spaces, similar to those of manifolds and orbifolds. As far as we know, such a description has not been established yet.
As explained above, the following is obtained as an immediate corollary.

\begin{corollary}\label{separated}
A separated stack $[M/G]$ is locally isomorphic to the stack underlying a linear action of a compact group.
\end{corollary}

The first proof of Zung's Theorem appeared in \cite{zung}, see also \cite{dz}. 
In the remaining of this section we overview the proof presented in \cite{cs}.
To begin with, we establish the following reduction, which is a strengthened version of \cite[2.2]{cs}.

%It essentially consists of three steps. Firstly, by restricting the Lie groupoid $G\toto M$ to a small open $x\in U\subset M$ we can assume that $G_U\toto M_U$ has the same objects, arrows, source and unit map as the action groupoid $G_x\ltimes T_xM\toto T_xM$. Secondly, we smoothly deform the other structural maps into their linearizations. Lastly, we prove that along this deformation the isomorphism type of a neighborhood of $0$ does not change.

\begin{proposition}\label{step1}
Given $G$ and $x$ as in \ref{zung}, there exists an open $x\in U \subset M$ and a diffeomorphism
$G_U\cong G_x\times \R^n$ extending the obvious one $G_x\cong G_x\times 0$, and such that the source  corresponds to the projection $G_x\times\R^n\to\R^n$ and the units to $1\times\R^n$.
\end{proposition}

\begin{proof}
To start with, by restricting $G$ to a small ball-like open around $x$, it is clear that we can assume $M=\R^n$ and $x=0$.

\smallskip

Moreover, we can assume that $G\subset G_0\times\R^n$, with source the projection. In fact, since $0$ is a fixed point we have $s^{-1}(0)=G_0$ and by the structure theorem for submersions (cf. \ref{strusubm}) there is an open $V$, $G_0\subset V\subset G$, on which the source looks as a projection. Since the anchor $\rho$ is proper at $0$ the open $V$ must contain a tube $G_W=\rho^{-1}(W\times W)\subset V$ (cf. \ref{charpatx}) and we can of course take $W\cong\R^n$.

\smallskip

Since $su=\id$ and $s$ is just the projection, the unit map $u:\R^n\to G_0\times\R^n$ can be written as $v\mapsto (u_1(v),v)$. Thus, in order to associate the units with the points $(1,v)$, we just need to compose the inclusion $G\subset G_0\times\R^n$ with the diffeomorphism $G_0\times\R^n\to G_0\times\R^n$, $(g,v)\mapsto(gu_1(v)^{-1},v)$.

\smallskip

To conclude we need to construct an open $x\in U\subset\R^n$ such that: it trivializes the source map, it is saturated, and it is diffeomorphic to $\R^n$.
Any small enough open trivializes the source, for $s:G\subset G_0\times\R^n\to\R^n$ is proper at $0$. 
Moreover, since $0$ is a stable orbit (cf. \ref{stable2}) there are arbitrary small saturated opens. The saturation of a set $W$ can be written as $ts^{-1}(W)$. The problem thus is how to take $W$ so as to get $ts^{-1}(W)\cong\R^n$. Next we provide an argument whose details are left to the reader.

\smallskip

Given $W$ a small open trivializing the source, its saturation can be written as
$$t(s^{-1}(W))=t(G_x\times W)=\bigcup_{g\in G_0}t(g\times W)$$
It is well-known that a star-shaped open in $\R^n$ is diffeomorphic to $\R^n$, thus it is enough to show that $t(g\times W)$ is starred at $0$ for all $g$.
Fixed $g$, the formula $x\mapsto t(g,x)$ defines a diffeomorphisms $\phi_g$ in a neighborhood of $0$ with inverse $\psi_g$.
The result now follows from the following lemma:

{\it If a smooth map $\psi:U\to U'$, $U,U'\subset\R^n$, $\psi(0)=0$ has $D\psi_0$ invertible then it maps small balls centered at $0$ to starred sets at $0$.}

In our case, since $G_0$ is compact, we can take the same ball for all $\psi_g$.
\end{proof}

In light of Proposition \ref{step1}, in order to prove Zung's theorem we can assume that 
$M=\R^n$, that $x=0$ is a fixed point with isotropy $G_0$, that $G=G_0\times\R^n$ and the source and unit maps are as follows.
$$(G\toto M)=(G_0\times\R^n\toto\R^n) \qquad s(g,v)=v \quad u(v)=(1,v)$$
Out of $G$ we  construct a new Lie groupoid $\tilde G$,
that can be seen as a 1-parameter family containing $G$ and the local model.
Its objects, arrows, source and unit maps are given by
$$(\tilde G\toto \tilde M)=(G_0\times\R^n\times\R\toto\R^n\times\R)
\qquad \tilde s(g,v,\epsilon)=(v,\epsilon)
\quad \tilde u(v,\epsilon)=(1,v,\epsilon)$$
whereas the other structural maps $\tilde t, \tilde m, \tilde i$ are defined by canonically deforming $t$, $m$, $i$ into their linearization. This is done by means of the following lemma.

\smallskip

{\it Let $A$ be a manifold and let $f:A\times\R^q\to \R$ be smooth and such that $f(x,0)=0$ for all $x$.
Then the function $\tilde f:A\times\R^q\times\R\to \R$ defined below is smooth.
$$\tilde f(x,y,\epsilon)=\begin{cases}\frac{1}{\epsilon}f(x,\epsilon y) & \epsilon\neq 0\\
\partial_yf|_{(x,0)} \cdot y& \epsilon=0\end{cases}$$}

\smallskip

For instance, the target map is defined by
$\tilde t(g,v,\epsilon)=(\frac{1}{\epsilon}t(g,\epsilon v),\epsilon)$ for $\epsilon\neq 0$ and $\tilde t(g,v,0)=(\partial_v t|_{(g,0)} \cdot v,0)$. The multiplication and inverse maps are defined similarly.
With these definitions, it is clear that $\tilde G\toto\tilde M$ is a well-defined Lie groupoid, and it is s-proper for $G_0$ is compact.

\begin{remark}
Let us explain how to see $\tilde G$ as a 1-parameter family. The projection
$$(\tilde G\toto \tilde M)\to(\R\toto\R) \qquad (g,v,\epsilon)\mapsto \epsilon \quad (v,\epsilon)\mapsto\epsilon$$
is a surjective map of groupoids, hence for each $\epsilon\in \R$ the fibers over $\epsilon$ give a new Lie groupoid
$\tilde G_\epsilon\toto\tilde M_\epsilon$
whose structural maps are induced by those of $\tilde G$. % These are the members of the family. 
For $\epsilon=1$ this is isomorphic to $G$, say $\tilde G_1\cong G$, and for $\epsilon=0$ this can be naturally identified with the action groupoid of the normal representation at $0$, say $\tilde G_0\cong G_0\ltimes\R^n$.
\end{remark}

The last step in proving \ref{zung} consists of proving that the family $\tilde G$ yields a trivial deformation near $0$. Such a trivialization is obtained by the flow of a multiplicative vector field.
A {\bf vector field} $X=(X_G,X_M)$ on the a Lie groupoid $G\toto M$ is just a pair of vector fields $X_G,X_M$ in $G,M$ respectively. Such a vector field is {\bf multiplicative} if it defines a groupoid map
$$(G\toto M)\to(TG\toto TM)$$
The flow of a multiplicative field is by Lie groupoid morphisms (cf. \cite{mx}), namely 
for each $\epsilon\in\R$ the $\epsilon$-flow is a groupoid map
$\phi_X^\epsilon:D\to G$ defined over an open subgroupoid $D\subset G$.

\begin{proposition}
The Lie groupoid $(\tilde G\toto \tilde M)$ constructed above admits a multiplicative vector field
$X$ such that $X\sim_\pi \partial_\epsilon$ and $X(g,0,\epsilon)=\partial_\epsilon$ for all $g$, $\epsilon$.
\end{proposition}

The proof of this can be consulted in \cite{cs}. Roughly, the idea is to lift $\partial_\epsilon$ to the obvious vector field $(g,v,\epsilon)\mapsto \partial_\epsilon$ on $\tilde G$ and then use an averaging argument to replace it by a multiplicative one $\tilde X$.

\bigskip

Once $X$ is constructed, the conclusion of \ref{zung} is routine. First, since the curves
$\gamma(t)=(g,0,t)$
are integral curves of $X$ we have that $G_0\times 0\times 0$ is contained in the open $D_1\subset \tilde G$ where the 1-flow $\phi_X^1$ is defined.
By the tube principle if $U$ is small enough then $G_0\times U\times 0\subset D_1$, and if in addition $U$ is an invariant ball-like open with respect to some $G_0$-invariant metric, we have an embedding
$$G_0\times U\times 0 \xto{\phi_X^1} G_0\times \R^n\times 1$$
whose image has to be of the form $G_V$, yielding an isomorphisms of Lie groupoids
$$(G_0\ltimes T_0M\toto T_0M)\cong(G_0\ltimes U\toto U)\cong(G_V\toto V).$$

\begin{remark}
It is well-known that the orbit space of a proper group action admits a smooth stratification (see eg. \cite{dk}).
It follow by Zung's Theorem that the orbit space $M/G$ of a proper groupoid is locally given by an action, and thus it locally admits smooth stratifications. 
A global stratification for the space $M/G$ is studied in \cite{ppt}.
\end{remark}

\begin{remark}
Corollary \ref{separated} suggests the existence of a chart description for separated stacks, similar to those of Chen's orbispaces \cite{chen} and Schwarz's Q-manifolds \cite{schwarz}.
We believe it would be interesting to investigate this and to better understand the relations among all these concepts.
We postpone this question to be treated elsewhere.
\end{remark}

\subsection{Linearization}

The Linearization Theorem \ref{linearization} unifies many linearization results such as Ehresmann Theorem for submersions, the Tube Theorem for group actions, and Reeb stability for foliations. 
It has becomed a milestone of the theory.

\bigskip

Given $G\toto M$ a Lie groupoid and $O\subset M$ an orbit, we can see $G_O\toto O$ as a subgroupoid of $G\toto M$ and also as the zero-section on the normal representation $NG_O\toto NO$.

$$(G\toto M) \from (G_O\toto O) \to (NG_O\toto NO)$$

The {\it linearization problem} consists in determine whether if this two groupoids are isomorphic in suitable neighborhoods. More precisely, $G$ is {\bf linearizable} at $O$ if there are open sets  $O\subset U\subset M$ and $O\subset V\subset NO$ and an isomorphism between the restrictions

$$(G_U\toto U) \cong ((NG_O)_V\toto V)$$ 

The linearization is called {\bf strict} if both $U,V$ can be taken to be saturated, and {\bf semistrict} if only $V$ can be taken saturated.
Within this language, Zung's Theorem \ref{zung} asserts that a proper groupoid can be linearized at a fixed point, and that the linearization is semistrict.

\begin{theorem}[Zung, Weinstein]\label{linearization}
If $G\toto M$ is proper at $x\in M$ then $G$ is linearizable at the orbit $O=O_x$. 
\end{theorem}

\begin{proof}
We know that the groupoid, restricted to a saturated open $O\subset U\subset M$, is equivalent to the restriction to a slice $G_S\toto S$ (cf. \ref{slice.chart}). Now, by Zung's Theorem \ref{zung}, we can assume that the restrict to the slice is isomorphic to the action groupoid of the normal representation $G_x\action N_xO$.

The equivalence $(G_U\toto U)\sim (G_x\ltimes N_xO\toto O)$ can be realized by a principal bibundle
$$\begin{matrix}
   G_x\times N_xO & \action & P & \raction & G_U\\
  \downdownarrows & \swarrow & & \searrow & \downdownarrows \\
N_xO & & & & U\end{matrix}$$

On the other hand, the inclusion $N_xO\to NO$ of the fiber into the vector bundle induces another equivalence $(G_x\ltimes N_xO\toto N_xO)\sim (NG_O\toto NO)$ and thus we have another principal bibundle
$$\begin{matrix}
   G_x\times N_xO & \action & P' & \raction & NG_O\\
  \downdownarrows & \swarrow & & \searrow & \downdownarrows \\
N_xO& & & & NO\end{matrix}$$

It is easy to check that the central fibers $P_0$, $P'_0$ of the submersions $P\to N_xO$, $P'\to N_xO$ are in fact equal, and furthermore we can identify $P'$ with the product $P_0\times N_xO$ as $G_x$-spaces.
It follows from \ref{gauge&bibundles} that in order to establish the desired isomorphism it is enough to show that $P$ and $P'$ are isomorphic as $G_x\ltimes N_xO$-bundles in a neighborhood of the central fibers.

Now, a principal $G_x\ltimes N_xO$-bundle is the same as a free proper action $G_x\action P$ and an equivariant submersion $P\to N_xO$. Thus the result follows from an equivariant version of \ref{strusubm}. Just pick a $G_x$-invariant metric on $P$ which makes the submersion $P\to N_xO$ Riemannian. Such a metric can be constructed first by taking an invariant metric on $P$, using it to define an invariant {\it horizontal direction} (the orthogonal to the fibers) and then redefine the metric on the horizontal direction by lifting an invariant one in $M$.  Then the exponential map associated to this metric provides a diffeomorphism $V\cong V'$ between open subsets $P_0\subset V\subset P$ and $P'_0\subset V'\subset P'$ compatible with the action and the submersion.
\end{proof}

We can think of this linearization theorem as a variant of theorem \ref{strusubm} applied to the presentation $M\to [M/G]$. In fact, when $G\toto M$ is proper without isotropy we have seen that $[M/G]$ is in fact a manifold, $M\to [M/G]$ is a submersion, and in this case both theorems state the same. From this point of view, it makes sense to ask which would be the corresponding version of Ehresmann's theorem \ref{etheorem}. It turns out that the properness of the map $M\to [M/G]$ at a point $[x]$ can be expressed as the $s$-properness of $G\toto M$ at $x$. 

\begin{corollary}\label{strict linearization}
If $G\toto M$ is $s$-proper at $x$ then $G$ is strictly linearizable at $O_x$.
\end{corollary}

\begin{proof}
Write $N=N_xO$, and using notations of previous theorem, we have a diagram
$$\xymatrix{ P \ar[d] & P\times_NP \ar[l] \ar[r] \ar@{}[dr]|(.7){\lefthalfcap}
\ar@{}[dl]|(.7){\righthalfcap}\ar[d] & G_U \ar[d]\\ N & \ar[l] P \ar[r] & U}$$
on which the left square is clearly a pullback and, since the right square can be regarded as a principal $G_x$-bundle map $(P\times_NP\to G_U)\to (P\to U)$, it is a pullback as well.

Properness is stable under base-change and $P\to U$ is a submersion. Thus, the fact of $s:G_U\to U$ being proper at $x$ implies that $P\to N$ is proper at $0$. We can finally apply the tube principle (cf. \ref{etheorem}) and shrink a linearizable open $P_0\subset V\subset P$ provided by \ref{linearization} to a saturated one. The result now follows.
\end{proof}

\begin{example}
A nice example to understand the difference between strict and non-strict linearization is the groupoid arising from the projection $M\subset S^1\times\R\to \R$, where $M$ is obtained by removing a point over $0$. The linear model around the orbit over $0$ does not perceive that the nearby orbits in the original groupoid are in fact compact.
\end{example}

\begin{remark}
In \cite{riemannian} we generalize Theorem \ref{linearization} by establishing linearization not only around orbits but around any saturated embedded submanifold. This may be regarded as the existence of tubular neighborhoods for separated stacks. 
\end{remark}

\begin{remark}
Given an action $G\action M$ of a compact group on a manifold, the action groupoid $G\times M\toto M$ is s-locally trivial and thus can be strict linearizable (cf. \ref{strict linearization}). This gives a tube around the orbit on which the action can be describe by means of the behavior on the orbit and on a slice. This is the well-known {\it Tube Theorem} (cf. \cite[2.4.1]{dk}). The tube theorem remains valid not only for actions of compact groups but for proper actions of Lie groups in general. This shows that condition in \ref{strict linearization} is sufficient but not necessary, and as far as we know a characterization of strict linearizable groupoids is still not known. Our guess is that s-local triviality should be enough.
\end{remark}

\medskip

\address{Matias L. del Hoyo, 
Dto. de Matem\'atica, Instituto Superior T\'ecnico, Av. Rovisco Pais, 1049-001 Lisboa, Portugal}

\medskip

\email{mdelhoyo@math.ist.utl.pt}

\url{http://mate.dm.uba.ar/$\sim$mdelhoyo/}

\end{document}